\documentclass[11pt,reqno]{amsart}
\usepackage{amsmath}
\usepackage[margin=1.0in]{geometry}
\usepackage{amscd,amsmath,amsxtra,amsthm,amssymb,stmaryrd,xr,mathrsfs,mathtools,enumerate,commath, comment, enumitem, graphicx}
\usepackage{amsfonts, amssymb, amsthm, amsmath, braket, hyperref, mathtools, comment, mathrsfs, enumitem, cleveref}
\usepackage{stmaryrd}
\usepackage{orcidlink}
\usepackage{multirow}
\usepackage{xcolor}
\usepackage{commath}
\usepackage{comment}
\usepackage{tikz-cd}
\usepackage{tkz-graph}
\usepackage{colonequals}
\usepackage{hyperref}
\usepackage{graphicx}
\usepackage{bigstrut}
\usepackage{longtable, multirow, array}
\usepackage{enumitem}

\renewcommand{\iff}{\text{ if and only if }}
\renewcommand{\pmod}[1]{\ (\mathrm{mod}\ #1) }

\newcommand{\fF}{\mathfrak{F}}

\usepackage{longtable} 
\usepackage{pdflscape} 
\usepackage{booktabs}
\usepackage{hyperref}
\definecolor{vegasgold}{rgb}{0.77, 0.7, 0.35}
\definecolor{darkgoldenrod}{rgb}{0.72, 0.53, 0.04}
\definecolor{gold(metallic)}{rgb}{0.83, 0.69, 0.22}
\hypersetup{
	colorlinks=true,
	linkcolor=darkgoldenrod,
	filecolor=brown,      
	urlcolor=gold(metallic),
	citecolor=darkgoldenrod,
}

\usepackage[all,cmtip]{xy}
\DeclareFontFamily{U}{wncy}{}
\DeclareFontShape{U}{wncy}{m}{n}{<->wncyr10}{}
\DeclareSymbolFont{mcy}{U}{wncy}{m}{n}
\DeclareMathSymbol{\Sh}{\mathord}{mcy}{"58}
\usepackage[T2A,T1]{fontenc}
\usepackage[OT2,T1]{fontenc}

\usepackage{tikz}
\usetikzlibrary{shapes.geometric}
\usetikzlibrary{decorations.markings}
\tikzset{every loop/.style={min distance=10mm,looseness=10}}
\tikzstyle{vertex}=[auto=left,circle,minimum size=1pt,inner sep=0pt]

\newtheorem{theorem}{Theorem}[section]
\newtheorem{lemma}[theorem]{Lemma}

\newtheorem{definition}{Definition}[section]
\newtheorem{corollary}[theorem]{Corollary}

\newtheorem{proposition}[theorem]{Proposition}
\newtheorem{question}{Question}[section]

\newtheorem*{theorem*}{Theorem}
\newtheorem*{ass*}{Assumption}
\newtheorem*{remark}{Remark}

\newcommand{\cF}{\mathcal{F}}

\newcommand{\bA}{\mathbb{A}}

\newcommand{\Z}{\mathbb{Z}}

\newcommand{\Q}{\mathbb{Q}}
\newcommand{\F}{\mathbb{F}}
\newcommand{\D}{\mathfrak{D}}

\newcommand{\N}{\mathbb{N}}

\newcommand{\op}[1]{\operatorname{#1}}

\numberwithin{equation}{section}

\begin{document}

	\title[Polynomial Compositions over Valued Fields]{On The Index of Polynomial Compositions over Valued Fields}
	
	\author[A. Jakhar]{Anuj Jakhar\, \orcidlink{0009-0007-5951-2261}}
\address[Jakhar, Kalwaniya]{Department of Mathematics, Indian Institute of Technology Madras, Chennai, Tamil Nadu, India-600036}
\email{anujjakhar@iitm.ac.in;ravikalwaniya3@gmail.com}
	
	\author[R.~Kalwaniya]{Ravi Kalwaniya\, \orcidlink{0009-0008-6964-5276}}

    \author[S. Laishram]{Shanta Laishram\, \orcidlink{0000-0001-6948-294X}}
	\address[Laishram]{Indian Statistical Institute, New Delhi, India}
	\email{shanta@isid.ac.in}
	
\author[P.~Yadav]{Prabhakar Yadav}
\address[Yadav]{FENS, Sabancı Üniversitesi, Istanbul, Turkey}
\email{pkyadav914@gmail.com; prabhakar.yadav@sabanciuniv.edu}
	
	\subjclass[2020]{}

\maketitle
\vspace{-0.2in}
\begin{flushright}
    {\it{-- -- Dedicated to Professor Sudesh Kaur Khanduja on the occasion of her 75th birthday}}
\end{flushright}
\begin{abstract}
Determining whether an algebraic number field admits a power integral basis is a classical problem, but it can be difficult for fields defined by polynomial compositions and dynamical iterates. In this paper, we study the monogeneity of compositions $f(h(x))$, where $f(x)$ and $h(x)$ are monic polynomials over an arbitrary Krull valuation ring and $h(x)$ is a trinomial. We derive explicit formulas for the discriminant of the composition and use them to characterize when $f(h(x))$ generates a monogenic field. In particular, we relate the monogeneity of the composition to that of $f(x)$ and to explicit square-free conditions on the  critical values of $h(x)$. We further extend the results to polynomial iteration and obtain a criterion for the monogeneity of binomial iterates, yielding new infinite families of monogenic fields. Finally, we give quantitative results and illustrate our criteria with several examples.
\end{abstract}

\maketitle

\section{Introduction} \label{sec:intro}
Let $K$ be a number field of degree $n$ over $\mathbb{Q}$ and $\mathbb{Z}_K$ denote its ring of integers. Since $\mathbb{Z}_K$ is a free $\mathbb{Z}$-module of rank $n$, it admits an integral basis $\{w_1,\ldots,w_n\}$. A natural question is whether this integral basis can be chosen in the particularly simple form $\{1,\alpha,\ldots,\alpha^{n-1}\}$ for some $\alpha\in\mathbb{Z}_K$. A number field $K$ is said to be \emph{monogenic} if such an element exists, or equivalently, if $\mathbb{Z}_K=\mathbb{Z}[\alpha]$ for some $\alpha\in\mathbb{Z}_K$. In this case, $\{1, \alpha, \ldots, \alpha^{n-1} \}$ is called a power integral basis of $K$. The existence of a power integral basis provides a particularly convenient description of the ring of integers and simplifies several arithmetic computations in number fields, including the determination of discriminants, integral bases and ideals. Despite this natural formulation, not every number field is monogenic, and deciding whether a given number field admits a power integral basis is a longstanding and difficult problem in algebraic number theory, going back to a question of Hasse~\cite{Hasse}.

Dedekind showed that the discriminant\footnote{The discriminant of a monic polynomial of degree $n$ having roots $\theta_1,\theta_2,\ldots,\theta_n$ is defined as $\prod_{1\leq i<j\leq n}(\theta_i-\theta_j)^2$.} $\D_f$ of a monic irreducible polynomial $f(x)$ and the discriminant $d_K$ of number field $K$ generated by a root $\theta$ over $\Q$ are related by
\begin{equation} \label{formula}
\D_f=[\Z_K:\Z[\theta]]^2d_K.
\end{equation}
Thus, if $\D_f$ and the index $[\Z_K:\Z[\theta]]$ are known, then the field discriminant $d_K$ can be determined from \eqref{formula}. In 1962, Swan \cite{Swan} was the first to give a discriminant formula for arbitrary trinomials of the form $x^n+ax^m+b$ in $\Z[x]$, using properties of the resultant. In 2019, Otake and Shaska~\cite{MR3916735} computed discriminant of polynomial $x^n +t(x^2 +ax+b)$ using  Bezoutians of polynomials. In 2021, Krystian~\cite{MR4563737} provided a discriminant formula for quadrinomials of the form
$x^n +ax^k +bx+c$, where  $k\in \{2,3, n-1\}$. Since then, discriminant formulas have been obtained for many other families of polynomials. For a comprehensive account of problems concerning discriminants and related problems, we refer the reader to the survey article of Ga\'{a}l \cite{Gaal2024}. Despite being an oldest problem, finding an effective formula for the discriminant of an arbitrary quadrinomial, particularly one that can be computed efficiently for any given quadrinomial, even if the polynomials has larger degree, still remains a challenging problem. To better understand this broader class of polynomials, particularly in higher degrees, a natural direction is to consider the following question.

\begin{question}
Let $f(x)$ and $g(x)$ be arbitrary monic polynomials. What is the discriminant of the composition $f(g(x))$ in terms of their coefficients and exponents?
\end{question}

In this direction, we obtain an explicit formula for the discriminant of the composition of arbitrary polynomial with trinomials. This result is of independent interest, as it expresses the discriminant of the composition in terms of the discriminant of the original polynomial along with the coefficients and exponents of the trinomials. In particular, once the discriminant of $f(x)$ is known, the discriminant of $f(h(x))$ can be computed explicitly.
\begin{theorem}\label{Discriminant}
Let $F$ be a field and $f(x) \in F[x]$ be a monic polynomial of degree $n$ and  $h(x) = x^m + a x^k + b \in F[x]$, where $\gcd(m, k) = t$, $m = m_1 t$, $k = k_1 t$ and $m - k = t$. Suppose that $f\circ h$ is irreducible over $F$. Then the discriminant of the composition  $f\circ h$ is given by
\[
\D_{f \circ h} = (-1)^{\,\frac{nm(nm+2k+n-4)}{2}} \, \D_f^{m} \,f(b)^{k - 1} \, m^{nm} \, f\left(G\left(\frac{-a k_1}{m_1}\right)\right)^t,
\]
where $G(x)=x^{m_1}+ax^{k_1}+b$.
\end{theorem}

Given the discriminant formula, it is clear from \eqref{formula} that the index $[\mathbb{Z}_K:\mathbb{Z}[\theta]]$ plays an important role in determining the field discriminant. This naturally leads to the question of when this index is equal to $1$. At this point, it is important to distinguish between the monogeneity of a number field and monogeneity of the polynomial defining it. A number field $K$ is called \textit{monogenic} if its ring of integers $\Z_K$ admits a power integral basis, i.e., $\Z_K=\mathbb{Z}[\alpha]$ for some algebraic integer $\alpha\in K$. For example, every quadratic number field is monogenic. On the other hand, if $f(x)\in\mathbb{Z}[x]$ is a monic irreducible polynomial with a root $\beta$, one can ask whether the order $\mathbb{Z}[\beta]$ is the full ring of integers of $K=\mathbb{Q}(\beta)$, or equivalently, whether $[\mathbb{Z}_K:\mathbb{Z}[\beta]]=1$. In this case, we say that the polynomial $f(x)$ is monogenic.

For example, consider $K=\mathbb{Q}(\sqrt{5})$. Since $K$ is quadratic, it is monogenic, with $\Z_K=\mathbb{Z}[\frac{1+\sqrt{5}}{2}]$. The minimal polynomial $g(x)=x^2-x-1$ of the generator therefore gives a power integral basis and is monogenic. Whereas, $f(x)=x^2-5$ defines the same number field, but its root $\sqrt{5}$ generates the proper suborder $\mathbb{Z}[\sqrt{5}]$, and $[\Z_K:\mathbb{Z}[\sqrt{5}]]=2$. Thus, although $K$ is monogenic, the polynomial $x^2-5$ defining it is not monogenic.

In 1989, Ga\'al and Schulte~\cite{MR979943} computed all power integral bases of cubic fields with bounded discriminant by applying Baker's effective method and the reduction procedure of Baker and Davenport. Further, Ga\'al~\cite{MR1229673, MR1333313} provided power integral basis for an infinite family of quartic and totally real cyclic sextic fields, respectively, and with Pohst~\cite{MR1423074}, they provided power integral basis for  totally real cyclic quintics field generated by Lehmer's parametric polynomial. In 1999, Ga\'al and Gy\H ory~\cite{MR1703860} provided an algorithm for solving index form equations in totally real quintic fields with Galois groups $M_{20}$ (metacyclic group of degree 5), $A_5$ or $S_5$. On the similar lines, Bilu, Ga{\'a}l and Gy{\H o}ry~\cite{MR2102808} studied index form equations in sextic fields to provide explicit resolution in the context of power integral bases.
 
For a comprehensive overview of problems related to monogenity, as well as further developments in the effective and algorithmic theory of monogenic number fields, we refer the reader to the survey articles by Ga\'al~\cite{Gaal2024} and Evertse and Gy\H{o}ry~\cite{EvertseGyory2025}.

Moving beyond quadratic extensions, the assumption that every algebraic number field is monogenic was disproved by Dedekind~\cite{RD1878} in 1878. He gave the first explicit example of a non-monogenic number field, generated by a root of the cubic $x^3-x^2-2x-8$, by showing that it has a common index divisor\footnote{Let $K=\mathbb{Q}(\theta)$ be a number field with ring of integers $\Z_K$. The \emph{index} of the element $\theta$ (or of its minimal polynomial $f$) is the group index $\operatorname{ind}(\theta)=[\Z_K:\mathbb{Z}[\theta]]$, so $f$ is monogenic precisely when $\operatorname{ind}(\theta)=1$. The \emph{index of the field}, denoted by $i(K)$, is defined as
\(i(K) = \{ \operatorname{ind}(\alpha) : \alpha \in \Z_K \} \). Any divisor of $i(K)$ is called as the \textit{common index divisor} of $K$}. This discovery, that a prime can divide the index of every integral element, led to the development of more refined methods for determining integral bases. In this direction, Ore~\cite{MR1512440} introduced the theory of first-order $\phi$-adic Newton polygons, which provides a method for studying the contribution of a prime $p$ to the index. However, the first-order approach does not always give enough information to determine the $p$-adic valuation of the index, particularly when the corresponding residual polynomial has repeated irreducible factors. To overcome this difficulty, Montes~\cite{Montes1999}, in his PhD thesis, developed the theory of higher-order Newton polygons by introducing higher-order valuations and residual polynomials that allow the arithmetic information to be lifted successively. This theory was later developed into a systematic and effective computational framework by Guàrdia, Montes, and Nart~\cite{MR2861080,MR2833586}, providing an algorithmic procedure for determining the $p$-adic valuation of the index and, consequently, for constructing integral bases.


The Montes algorithm and higher-order Newton polygon techniques provide powerful tools for studying monogeneity and computing integral bases. However, their computational complexity can become significant as the degree of the defining polynomial increases, particularly for number
fields arising from iterated or composed polynomials. This has motivated the development of more efficient approaches for such families. For instance, Jakhar, Laishram, Srinivas, and Yadav~\cite{jakhar2025behaviour}
recently employed Newton polygon techniques to study polynomial iterates, obtaining more efficient criteria for the behaviour of indices and ramification in the resulting towers of number fields. Despite these developments, a direct application of existing computational methods to
highly composed polynomials may still involve substantial calculations. This motivates the search for structural criteria that can determine monogeneity and index behaviour more directly from the constituent polynomials.

For a comprehensive treatment of the effective theory of monogenic number fields, we refer the reader to the monograph of Evertse and Gy{\H o}ry \cite{MR3586280}. This work develops general effective methods and algorithms
for studying monogenic number fields, including effective finiteness results for power integral bases and their various applications and generalizations. For a more algorithmic and computational perspective, we refer to the monograph of Ga{\'a}l \cite{MR3970246}, which focuses on practical methods for determining power integral bases, together with explicit computations and numerical results for number fields of relatively small degree.
This also leads to a natural question concerning monogeneity:
\begin{question}
Let $f(x)$ and $g(x)$ be monic polynomials in $\mathbb{Z}[x]$. Under what conditions are both $f(x)$ and $f(g(x))$ monogenic?
\end{question}
The above question is interesting because it highlights how discriminants and indices of the underlying polynomials are related, and it also provides a way to construct new infinite families of monogenic number fields from known ones. However, the problem is rather subtle. When polynomials are composed, their arithmetic invariants do not behave in a straightforward manner, which makes the analysis technically challenging. Nevertheless, significant progress has been made in this direction. Harrington and Jones~\cite{JonesHarr} studied pairs of binomials 
$f(x) = x^n - a$ and $g(x) = x^m - b$ for which both $f(x)$ and composition $f(g(x))$ are monogenic. In particular, they obtained necessary and sufficient conditions for the monogeneity of $(x^m - b)^n - a$ in the cases $n = 2, 3$, and proposed a conjecture for higher degrees, that is, for $n \geq 4$. Later, Jakhar, Kalwaniya, and Yadav~\cite{J-K-Y} extended these results by establishing the necessary and sufficient conditions for $\mathbb{Z}_K = \mathbb{Z}[\theta]$, where $\theta$ is the root of an irreducible polynomial $(x^m - b)^n - a \in \mathbb{Z}[x]$ with $n \geq 2$ and $m \geq 1$. More recently, using Dedekind's index criterion, they~\cite{JakharkalwaniyaYadav} derived necessary and sufficient conditions for the monogeneity of binomial and trinomial composition. In 2025, Kaur, Kumar, and Remete \cite{KKR} established necessary and sufficient conditions for a monic irreducible polynomial $f(x)\in\mathbb{Z}[x]$ such that the power composition $f(x^k)$, with $k\in\mathbb{Z}_{>0}$, is monogenic. 

In this direction, we consider compositions of an arbitrary monic polynomial with the trinomial $h(x)=x^m+ax^{m-1}+b$. Assuming that $f(h(x^t))$ is irreducible over $\Q$, our next theorem gives a complete characterization of the monogeneity of $f(h(x^t))$.
\begin{theorem}\label{thm:2.1}
    Let $f(x) \in \Z[x]$ be a monic polynomial of degree $n$ and $h(x) = x^m + a x^{m-1} + b \in \Z[x]$ be a monic polynomial of degree $m \geq 2$ such that $f(h(x^t))$ is irreducible over $\Q$. 
    Then $f(h(x^t))$ is monogenic if and only if: 
    \begin{itemize}
        \item [\textup{(1)}] $f(x)$ is monogenic;
        \item  [\textup{(2)}] $f(b)$ is squarefree whenever $m>2$. Further, if $m=2$ then $p^2 \nmid f(b)$ whenever $p$ divides both $ a f'(b)$ and $\D_{f(h(x^t))}$;
        \item [\textup{(3)}] $p^2\nmid m^{mn}\,f\left(h\left(\frac{-a(m-1)}{m}\right)\right)$ for every $p\nmid \gcd(a,m)$;
        \item [\textup{(4)}] $p$ does not divide index of $f(x^p+b)$ for every $p\mid \gcd (a, m)$ and $p^2\mid a$;
        \item[\textup{(5)}] For $p\mid \gcd(a,m)$ with $p\mid \mid a$, write $m=p^u \ell$ with $p\nmid \ell$. Then $f(x^{\ell}+b)$ and $\frac{f(x^m+b)-f(x^{\ell}+b)^{p^u}}{p}+\frac{a}{p}x^{m-1}f'(x^m+b)$ are coprime modulo $p$;
        \item [\textup{(6)}] $p$ does not divide index of $f(h(x^p))$ for every $p\mid t$.
    \end{itemize}
\end{theorem}

\begin{remark}
In the above theorem, we assume $m\geq 2$ since the case $m=1$ is immediate. Indeed, when $m=1$, we have $h(x)=x+c\in\mathbb{Z}[x]$, and the monogeneity of $f(x)$ and that of $f(h(x))=f(x+c)$ are equivalent. Further, for $r>1$, the monogeneity of $f(h(x^r))=f(x^r+c)$ can be obtained as a special case of the above theorem by taking $m=r$, $a=0$, $t=1$, and $b=c$.
\end{remark}

\begin{remark}
We prove the Theorem \ref{thm:2.1} in the more general setting of Krull valuations in Section~\ref{sec:monogenity of composition} see Theorem~\ref{VT2.1}. In fact, the results stated below can also be obtained in this general setting of Krull valuations by applying the results proved in Section~\ref{sec:monogenity of composition}. However, as the applications over $\mathbb{Q}$ are easy to understand and the proofs in the general setting would be essentially repetitive, we state and prove the below results only over $\mathbb{Q}$.
\end{remark}
The following corollary follows from Theorem~\ref{thm:2.1} and will be useful for studying the monogenity of compositions of binomial with arbitrary polynomials.
\begin{corollary}\label{T2.1}
    Let $f(x) \in \Z[x]$ be a monic polynomial of degree $n$, and let $h(x) = x^m + a x^{m-1} + b \in \Z[x]$ be a monic polynomial such that $f(h(x))$ is irreducible over $\Q$. 
    Suppose that for every prime divisor $p$ of $\gcd(a,m)$, we have $p^2 \mid a$.
    Then $f(h(x^t))$ is monogenic if and only if: 
    \begin{itemize}
        \item [\textup{(1)}] $f(x)$ is monogenic;
        \item  [\textup{(2)}] $f(b)$ squarefree whenever $m>2$. Further, if $m=2$ then $p^2 \nmid f(b)$ whenever $p \mid a f'(b)$;
        \item [\textup{(3)}] $p^2\nmid m^{mn}\,f\left(h\left(\frac{-a(m-1)}{m}\right)\right)$ for every $p\nmid \gcd(a,m)$;
        \item [\textup{(4)}] $p$ does not divide index of $f(x^p+b)$ for every $p\mid \gcd (a, m)$;
        \item [\textup{(5)}] $p$ does not divide index of $f(h(x^p))$ for every $p\mid t$.
    \end{itemize}
\end{corollary}
The following corollary provides a necessary and sufficient condition for the monogenity of the composition of arbitrary polynomial $f(x)$ with a binomial. Taking $a=0$ and $t=1$ in above corollary yields the following result.
\begin{corollary}
    Let $f(x) \in \Z[x]$ be a monic polynomial of degree $n$ such that $f(x^m+b)$ irreducible over $\Q$. Then $f(x^m+b)$ is monogenic if and only if 
    \begin{itemize}
        \item [\textup{(1)}] $f(x)$ is monogenic;
         \item  [\textup{(2)}] $f(b)$ squarefree;
         \item [\textup{(3)}] $p$ does not divide index of $f(x^p+b)$ for every $p\mid m$.
    \end{itemize}
\end{corollary}
We wish to point out that the above corollary generalizes the main result of paper by Kaur, Kumar and Remete \cite{KKR}, which is obtained by taking $b=0$.

Having understood the monogeneity of polynomial compositions, a natural next step in studying number fields of higher degree is to consider iterates of a polynomial and ask whether the resulting iterated polynomials remain monogenic. These observations leads to the following natural question concerning the monogeneity of polynomial iterates:
\begin{question}
Let $f(x)$ and $g(x)$ be monic polynomials in $\mathbb{Z}[x]$. Under what conditions is $(f\circ g^N)(x)$ monogenic?
\end{question}

Partial results in this direction are known for several important families of polynomials. Iterates of quadratic polynomials of the form $f(x)=x^2-c$ have been studied in \cite{castillo,Li}, where necessary and sufficient conditions for their monogeneity were obtained. In 2023, Smith \cite{Smith23} established a necessary and sufficient condition for $f^n(x)$ to be monogenic when $f(x)=x^p-a\in\mathbb{Z}[x]$ has prime degree. More recently, Sharma, Sarma and Laishram~ \cite{S-L-S} obtained a necessary and sufficient condition for the iterates of a $p$-Eisenstein polynomial $f(x)\in\mathbb{Z}[x]$ to be monogenic, with their criterion formulated in terms of the square-freeness of the product $\prod_{f'(a_i)=0}f^n(a_i)$. Smith and Wolske \cite{hanson} also investigated monogeneity and prime splitting in extensions generated by roots of iterates of quadratic polynomials $f(x)=x^2+bx+c\in\mathbb{Z}[x]$.

More recently, the study of monogenity of polynomial iterates has been considered from a broader dynamical perspective. K{\"o}nig, Smith and Wolske~\cite{konig2024critical} established necessary and sufficient conditions, in terms of the critical orbit of a monic irreducible polynomial $f(x)\in\Z_K[x]$, for its iterates to be monogenic. They further used these results to construct dynamically monogenic polynomials that give rise to infinite towers of monogenic number fields, for e.g., monogeniety of tower of number fields generated by $T_{d}^n(x) - a$, where $T_d(x)$ is a Chebyshev polynomial of degree $d$.

In the dynamical setting, while the existing literature has mainly focused on the iteration of a single polynomial, we consider a more general compositional framework. To the best of our knowledge, this is the first study to characterize monogeneity in infinite dynamical towers arising from the composition of distinct polynomials. In particular, we establish algebraic conditions under which composite iterates of the form $f\circ g^n(x)$ remain monogenic. Our next result gives a necessary and sufficient conditions for the monogeneity of $f\circ g^N$, where $f(x)$ is an arbitrary polynomial and $g(x)=x^{mt}+ax^{(m-1)t}+b$ is a trinomial. 
{}

\begin{theorem}\label{thm:2.6}
Let $f(x)\in \mathbb{Z}[x]$ be a monic polynomial of degree $n$ and $h(x)=x^{m}+ax^{m-1}+b \in \mathbb{Z}[x].$ Define $\kappa= t\gcd(a,m)$ and $H(x)=h(x^t)$. Assume that, for every prime $p\mid \kappa$, each irreducible factor of $f(x)$ over $\mathbb{F}_p$ has multiplicity at least two. Further, assume that $p^2\mid a$ for every prime $p\mid \gcd(a,m)$. For some $N \in \N$, suppose that $f\circ H^{N}(x)$ is irreducible over $\mathbb{Q}$. Then $f\circ H^{N}(x)$ is monogenic if and only if the following conditions hold:
\begin{itemize}
    \item[\textup{(1)}] $f(x)$ is monogenic;
    \item[\textup{(2)}] For every prime $p\mid \kappa$,  we have $p\nmid \operatorname{ind}(f(x^p));$
    \item[\textup{(3)}] $(f\circ H^{i})(0)$ are square-free for every $1\leq i\leq N$ whenever $m>2$. Further, if $m=2$ then $p^2 \nmid (f\circ H^i)(0)$ whenever $p \mid a (f\circ H^{i-1})'(H(0))$ for every $1\leq i\leq N$;
    \item[\textup{(4)}]  For each  $j\in\{0,1,\ldots,N-1\}$,
    \[ p^2 \nmid m^{mn(tm)^j} \,(f\circ H^{j}) \!\left(
    h\!\left(\frac{-a(m-1)}{m}\right) \right) \text{ whenever } p\nmid \gcd(a, m). \]
\end{itemize}
\end{theorem}
We note the following important fact: if $f\circ H^N$ is monogenic, then all the intermediate compositions are also monogenic. In other words, $f\circ H^i$ is monogenic for every $0\leq i\leq N$. The following corollary is immediate from Theorem~\ref{thm:2.6} by taking $a=0$ and $t=1$.
\begin{corollary}\label{cor:2.7}
Let $f(x) \in \mathbb{Z}[x]$ be a monic polynomial of degree $n$ and $g(x) = x^m + b \in \Z[x]$ be such that $f \circ g^N$ is irreducible over $\mathbb{Q}$ for some $N \in \N$. Assume that for every prime $p \mid m$, each irreducible factor of $f(x)$ over $\mathbb{F}_p$ has multiplicity at least two. Then $(f\circ g^N)(x)$ is monogenic if and only if the following conditions hold:
\begin{itemize}
    \item[\textup{(1)}] $f(x)$ is monogenic;
    \item[\textup{(2)}] $p$ does not divide the index of $f(x^p)$ for every prime $p \mid m$;
    \item[\textup{(3)}] $(f\circ g^i)(0)$ are squarefree for every $1 \leq i \le N$.
\end{itemize}
\end{corollary}
The main beauty of the above result is that, after the first composition $f \circ g$ being monogenic, the monogeneity of the entire tower of number fields depends only on the squarefreeness of $(f\circ g^i)(0)$, with no additional conditions required. Our next result removes the need to check the repeated-root condition appearing in the above corollary and provides a direct way to construct new monogenic polynomials from any given monogenic polynomial. In particular, once we start with a monogenic polynomial, the result gives a systematic construction of further monogenic polynomials without having to verify the repeated-root condition for each prime divisor $p$ of $m$.

\begin{theorem}\label{thm:2.8}
Let $f(x) \in \mathbb{Z}[x]$ be a monic polynomial of degree $n$ and  $g(x) = x^m + b \in \mathbb{Z}[x]$. Let $\ell$ be an integer divisible by  $\operatorname{rad}(m)$ and define $F(x) = f(x^{\ell})$. Assume that $\mathcal{F}(x)=F(g(x))$ is irreducible over $\mathbb{Q}$. Then $\mathcal{F}(x)$ is monogenic if and only if the following conditions hold:
\begin{itemize}
    \item[\textup{(1)}] $f(x)$ is monogenic;
    \item[\textup{(2)}] $f(0)$ and $f(b^{\ell})$ are squarefree;
    \item[\textup{(3)}] $p$ does not divide the index of $f(x^p)$ for every prime $p \mid \ell$.
\end{itemize}
\end{theorem}
\begin{theorem}\label{thm:2.9}
Let $f(x)=x^n+a$ and $g(x)=x^m+b$ be polynomials in $\mathbb{Z}[x]$. Assume that $\operatorname{rad}(m)\mid n$. Given any $N\in\mathbb{N}$, if $f \circ g^N$ is irreducible then the composition $f\circ g^{N}$ is monogenic if and only if the following conditions hold:
\begin{itemize}
    \item[(1)] $(f\circ g^i)(0)$ square free for $0\leq i\leq N$;
    \item[(2)] $p^2\nmid [(-a)^p+a]$ for every $p\mid n$.

\end{itemize}
\end{theorem}
Taking $m=n$ and $b=a$ in Theorem~\ref{thm:2.9} recovers the main result of Sharma, Sarma and Laishram~\cite{S-L-S}.
\begin{remark}
   For $f(x)=x^n+a$ and $g(x)=x^m+b$ the polynomial $f$ is $g$-stable whenever there exists a prime $p$ such that $p\mid \gcd(a,b)$ and $p^2\nmid a.$
\end{remark}

We see from \eqref{formula} that if $\D_f$ is squarefree, then $f(x)$ is monogenic. Using this observation, for positive real numbers $A_1$ and $A_2$ with $A_1 > A_2^{1+\D_0}$, and $5 \leq n \equiv 1 \pmod{4}$, Mukhopadhyay,  Murty and Srinivas~\cite{MMK} showed that
\[
\#\{ (a_1,a_2) \in \Z \times \Z : A_i \leq |a_i| \leq 2A_i, \; x^n + a_1x + a_2 \text{ is monogenic} \} \gg A_1A_2.
\]

This result was proved under the assumption of the $abc$-conjecture. In 2012, Kedlaya~\cite{kedlaya} provided a method to generate infinitely many monogenic polynomials, of any degree greater than one, whose discriminants are squarefree. Later, in \cite{Jones, L2}, Jones constructed further infinite families of monogenic polynomials with squarefree discriminants of various types.\\
In past years, several authors have obtained many interesting quantitative results on irreducible polynomials with non-square-free discriminants. Recently, Jones~\cite{L1} provided infinitely many monogenic polynomials of a given prime degree~$p$ whose discriminant is not squarefree. Part of this result relied on the generalized $abc$-conjecture for number fields. Using a method similar to that of Jones, Barman, Narode and Wagh~\cite{RAV} constructed another infinite family of monogenic polynomials of degree~$q$ with non-squarefree discriminants, where $q$ is a prime of the form $q = q_0 + q_1 - 1$, with $q_0$ and $q_1$ both prime. Furthermore, Jones and White~\cite{L3} gave an asymptotic formula for certain special types of monogenic trinomials.\\ 
In this direction, our next theorem provides a lower bound on the number of pairs of polynomials $f(x)$ and $g(x)$ such that both $f(x)$ and $f(g(x))$ are monogenic under the assumption of $abc$-conjecture.
\begin{theorem}\label{analytic}
Let $q$ be a prime. Let $n \ge 2$ and $m \ge 3$ be integers and  $a$ be a nonzero integer such that $q \mid a$ and $\gcd(a,m)=1$. Put $\epsilon = (-1)^{m+1}a^m(m-1)^{m-1}$ and, for $b \in \mathbb{Z}$, define \( h_b(x) = x^m + ax^{m-1} + b \).
Let $f(x) = x^n + c_{n-1}x^{n-1} + \dots + c_0 \in \mathbb{Z}[x]$ be a monic, $q$-Eisenstein, and monogenic polynomial. Set
\begin{equation*}
    \mathcal{N}_f(b) = \sum_{j=0}^n c_j(m^m b + \epsilon)^j m^{m(n-j)} \in \mathbb{Z}[b]; \quad R_f(y) = \frac{f(qy)}{q} \cdot \frac{\mathcal{N}_f(qy)}{q} \in \mathbb{Q}[y].
\end{equation*}
Suppose that for every prime $p \le 2n$ such that $p \neq q$ and $p \nmid m$, we have $p^2 \nmid R_f(0)$. If the $abc$-conjecture is true, then
\begin{equation*}
    \#\{1 \le b \le Y : q \mid b \text{ and } f(h_b(x)) \text{ is monogenic}\} \ge (\lambda + o(1))\left(\frac{Y}{q}\right),
\end{equation*}
where $\lambda = \prod_{p}\left(1 - \frac{\omega(p)}{p^2}\right) > 0$ and $\omega(p) = \#\{1 \le \alpha \le p^2 : R_f(\alpha) \equiv 0 \pmod{p^2}\}$. In particular, $f(x)$ and $\cF_b(x) = f(h_b(x))$ are simultaneously monogenic for a positive proportion of the integers $b$.
\end{theorem}

\begin{remark}
Note that, in the above theorem, the polynomial $\mathcal{F}_b(x)$ obtained after the composition is again monic, $q$-Eisenstein, and monogenic. Therefore, the same construction can be applied repeatedly. More precisely, starting with $\mathcal{F}_{b_0}(x)=f(x)$, one can construct a sequence of integers $b_k$ such that $\mathcal{F}_{b_k}(x)=\mathcal{F}_{b_{k-1}}\bigl(h_{b_k}(x)\bigr)$
and each polynomial $\mathcal{F}_{b_k}(x)$ remains monic, $q$-Eisenstein, and monogenic. This provides a recursive construction of monogenic polynomials of arbitrarily large degree and, consequently, of monogenic number fields of arbitrarily large degree.
\end{remark}

 \noindent \textbf{The structure of the paper is as follows.} This paper is organized into six sections, including the introduction. In Section~\ref{sec:disc_of_comp}, we present some preliminary results and compute the discriminant of polynomial compositions; see Theorem~\ref{Discriminant}. Section~\ref{sec:monogenity of composition} is devoted to results over valuation rings. We establish several results in a general setting and prove Theorem~\ref{VT2.1}. In Section~\ref{sec:proof over rationals}, we prove Theorems~\ref{thm:2.1}, \ref{thm:2.6}, \ref{thm:2.8}, and \ref{thm:2.9}. In Section~\ref{sec:analytic}, we prove Theorem~\ref{thm:2.9}, which establishes a lower bound for the number of pairs of polynomials $f(x)$ and $g(x)$ for which both $f(x)$ and $f(g(x))$ are monogenic. Finally, Section~\ref{sec:examples} contains examples illustrating our results.

\section{Discriminant of Composition of polynomials} \label{sec:disc_of_comp}
We begin this section with two lemmas that will be used in the proof of Theorem~\ref{Discriminant}. Let $\theta$ be a root of an irreducible monic polynomial $f(x)\in F[x]$. The following lemma relates the discriminant of $f(x)$ to the norm of $f'(\theta)$.
\begin{lemma} \cite[Proposition 12.9]{Morandi96}\label{norm-compute} Let $f(x) \in F[x]$ be monic and irreducible polynomial with $\deg(f)=n.$ Let $K=F(\alpha)$ be a field extension of $F$ where $\alpha$ is a root of $f(x)$. Then
	$$ \D_f=(-1)^{\binom{n}{2}}\,\mathcal{N}_{K/F} (f'(\alpha)).$$
\end{lemma}
Let $K$ be a finite extension of a field $F$. For any element $\alpha \in K$, computing the norm of $\alpha$ with respect to $K/F$ is closely related to computing the norm of $\alpha$ with respect to $F(\alpha)/F$, as given by the following formula.

\begin{lemma}\cite[Theorem~1.20]{SKM}\label {norm in extension} Let $K/F$ be an extension of degree $n$ and $\alpha$  be an element of $K$ with $[F(\alpha):F]=d$. Then $$ \mathcal{N}_{K/F}(\alpha)=\left( \mathcal{N}_{F(\alpha)/F}(\alpha)\right)^{n/d}. $$
\end{lemma}

\begin{proof}[Proof of Theorem~\ref{Discriminant}]
Let $\theta$ be a root of $f\circ h$ and $K=F(\theta)$. Using Lemma~\ref{norm-compute}, we have
\begin{align*}
    \D_{f\circ h } &= (-1)^{nm(nm-1)/2} \, \mathcal{N}_{K/F}\left((f(h(\theta)))'\right) \\
    &= (-1)^{nm(nm-1)/2} \, \mathcal{N}_{K/F} \left(f'(h(\theta))\right) \cdot \mathcal{N}_{K/F} \left(h'(\theta)\right) \quad \textit{(since norm is multiplicative)}.
\end{align*}
Now, applying Lemma~\ref{norm-compute} to $f(x)$ and substituting $h'(\theta) = \theta^{k-1}(m\theta^t + ak)$ into the expression above, we obtain
\begin{equation*}
    \D_{f\circ h } =(-1)^{\frac{nm(nm-1)+nm(n-1)}{2}}\, \D_f^{m} \, \mathcal{N}_{K/F}(\theta^{k-1}(m\theta^t + ak)).
\end{equation*}
Using multiplicative property of Norm, we have
\begin{equation}\label{disc f.h}
    \D_{f\circ h } =(-1)^{\frac{nm(nm+n-2)}{2}} \,\D_f^m \, \mathcal{N}_{K/F}(\theta)^{k-1}\,\mathcal{N}_{K/F}(m\theta^t + ak).
\end{equation}
Since the minimal polynomial of $\theta$ is $f(h(x))$, it follows that
\begin{equation} \label{norm of theta}
    \mathcal{N}_{K/F}(\theta) = (-1)^{nm} \, f(b).
\end{equation}
Next, we compute $\mathcal{N}_{K/F}(m\theta^t + ak)$. Let us denote $z = m\theta^t + ak$ and  $f(x) = x^n + a_{n-1}x^{n-1} + \cdots + a_1x + a_0$. Since $f(h(\theta)) = 0$, we substitute $\theta^t = \frac{z - ak}{m}$ to obtain
\begin{equation}
    f\left(\left(\frac{z - ak}{m}\right)^{m_1} + a\left(\frac{z - ak}{m}\right)^{k_1} + b\right) = \left(\left(\frac{z - ak}{m}\right)^{m_1} + a\left(\frac{z - ak}{m}\right)^{k_1} + b\right)^n + \cdots + a_0 = 0.
\end{equation}
Therefore, we see that $z$ is a root of the polynomial
\begin{equation}\label{minimal poly for z}
    h_1(x) = \left(G_1(x)\right)^n + a_{n-1} m^{m_1} \left(G_1(x)\right)^{n-1} + \cdots + a_0 m^{m_1n},
\end{equation}
where $G_1(x) = (x - ak)^{m_1} + am(x - ak)^{k_1} + bm^{m_1}$.
We now show that $h_1(x)$ is the minimal polynomial of $z$ over $F$.
 It is easy to verify that $F(z)=F(\theta^{t})$. Hence, to establish that the degree of the minimal polynomial of $z$ over $F$ is $m_1n$, it is sufficient to prove that $[F(\theta^t):F]=m_1n$. Now define a polynomial 
\[g(x)=[x^{m_1}+ax^{k_1}+b]^n+a_{n-1}[x^{m_1}+ax^{k_1}+b]^{n-1}+\cdots+ a_1[x^{m_1}+ax^{k_1}+b]+a_0\in F[x].\]
We will show that $g(x)$ is minimal polynomial for $\theta^t$, i.e., $g(x)$ is irreducible and $g(\theta^t)=0$. We have
\[ g(\theta^t)=[(\theta^t)^{m_1}+a(\theta^t)^{k_1}+b]^n+a_{n-1}[(\theta^t)^{m_1}+a(\theta^t)^{k_1}+b]^{n-1}+\cdots+ a_0=f(h(\theta))=0. \]
  Now suppose $g(x)$ is reducible say $g(x)=g_1(x)\, g_2(x)\in F[x]$ with degree $\deg(g_1),~ \deg (g_2) \geq 2$. Clearly $g(x^t)=f(h(x))$. Using this we obtain that $f(h(x))=g_1(x^t)\cdot g_2(x^t)$. This contradicts the irreducibility of $f(h(x)).$ Hence, $g(x)$ is irreducible. This proves that $h_1(x)$ is minimal polynomial for $z$. Thus, 
  \[\mathcal{N}_{K(z)/F}=(-1)^{m_1n}h_1(0).\] Using Lemma \ref{norm in extension}, we get
  \[\mathcal{N}_{K/F}(m\theta^t+ak)=\left((-1)^{m_1n}h_1(0)\right)^t.\]
Substituting value of $h_1(0)$ in above equation, we have
\begin{equation}\label{norm of z}
    \mathcal{N}_{K/F}\,(m\theta^t+ak)=(-1)^{\,mn}\,m^{\,mn}\,f\left(\left(\frac{- ak}{m}\right)^{m_1} + a\left(\frac{- ak}{m}\right)^{k_1} + b\right)^t.
\end{equation}
Therefore, the theorem follows from Lemma~\ref{norm-compute}, together with Equations~\eqref{disc f.h}, \eqref{norm of theta} and~\eqref{norm of z}.
\end{proof}

\section{Index of composed polynomial} \label{sec:monogenity of composition}
Let $p$ be a prime number and $\overline{f}(x) = \overline{g}_1(x)^{e_1}\cdots\overline{g}_r(x)^{e_r}$ be the factorization of the polynomial $\overline{F}(x)$ obtained on replacing each coefficient of $f(x)$ modulo $p$ into a product of powers of distinct irreducible polynomials over $\mathbb{Z}/p\mathbb{Z}$ with $g_i(x)$ belonging to $\mathbb{Z}[x]$ monic. In 1878, Dedekind proved that if $p$ does not divide $[\bA_K:\mathbb{Z}[\theta]]$, then $p\bA_K=p_1^{e_1} \cdots p_r^{e_r},$ where $p_1,\ldots,p_r$ are distinct prime ideals of $\bA_K$,
$p_i=p\bA_K+g_i(\theta) \bA_K$, with norm of $p_i$ equal to
$p^{\deg g_i}$ for $1\leq i\leq r$. Dedekind also gave a simple criterion, known as Dedekind Criterion, to verify when the condition $p\nmid[\bA_K:\mathbb{Z}[\theta]]$ is satisfied and proved that this holds if and only if for each $i$, either $e_i=1$ or $\overline{g}_i(x)$ does not divide $\overline{H}(x)$, where $H(x)=\frac{1}{p}\left(f(x)-g_1(x)^{e_1}\cdots g_r(x)^{e_r}\right)$ (cf. \cite[Theorem 6.1.4]{HC}). If $\mathbb{Z}_{(p)}$ denotes the localization of $\mathbb{Z}$ at the prime ideal $p\mathbb{Z}$, then using Lagrange's Theorem and Cauchy's Theorem for finite groups, one can quickly verify that the condition $p \nmid[\bA_K:\mathbb{Z}[\theta]]$ is equivalent to $\bA_K\subseteq\mathbb{Z}_{(p)}[\theta]$, which is the same as requiring that the ring $\mathbb{Z}_{(p)}[\theta]$ is integrally closed. With this interpretation, Dedekind Criterion was extended in a natural way by Khanduja and Kumar~\cite{MR2598906} replacing $\mathbb{Z}_{(p)}$ by any discrete valuation ring. In 2006, Ershov proved an analogous result for general valuation rings (cf. \cite{MR2455373, MR2598906}).

In order to state the generalized Dedekind Criterion proved by Ershov, we first need few definitions and notation. So, we now briefly recall the notions of Krull valuations and valuation rings that will be used throughout this section. We also fix the notation adopted in the sequel.

\begin{definition}
Let $K$ be a field and let $\Gamma$ be a totally ordered abelian group, written additively. A \emph{Krull valuation} on $K$ is a surjective map
\( v:K^\times\longrightarrow\Gamma \)
satisfying
\begin{enumerate}
\item [\textup{(1)}] $v(xy)=v(x)+v(y)$;
\item [\textup{(2)}] $v(x+y)\ge \min\{v(x),v(y)\}$ whenever $x+y\neq0$.
\end{enumerate}
The valuation is extended to all of $K$ by setting $v(0)=\infty$, where $\infty$ is a symbol satisfying $\gamma<\infty$ and $\gamma+\infty=\infty$ for every $\gamma\in\Gamma$. The ordered abelian group $\Gamma$ is called the \emph{value group} of $v$.
\end{definition}
\noindent Associated with $v$ is the valuation ring $R_v=\{x\in K:\,v(x)\ge0\},$ whose unique maximal ideal is $M_v=\{x\in K:\,v(x)>0\}.$ The field of fractions of $R_v$ is $K$.

\medskip

\noindent
\textbf{Notation.}
Throughout this section, $R$ is an integral domain (with unity) and $F$ is the field of fraction of $R$. $v$ denotes a Krull valuation of arbitrary rank on a field $F$, with valuation ring $R_v$, maximal ideal $M_v$, value group $\Gamma_v$ and residue field \( k_v=R_v/M_v. \) For an element $\xi \in R_v$, we write $\overline{\xi}$ for its residue class in $k_v$. If $g(x)=\sum_{i=0}^{n}a_i x^i\in R_v[x],$ then $\overline g(x)=\sum_{i=0}^{n}\overline a_i x^i\in k_v[x]$ denotes the polynomial obtained by reducing each coefficient of $g(x)$ modulo $M_v$.

Now we state the result of Ershov \cite{MR2455373} famously known as Generalized Dedekind criterion (GDC) for Krull valuation of arbittrary rank proved in \cite{MR2455373}.

\begin{theorem}[Generalized Dedekind Criterion (GDC)]
Let $v$, $R_v, M_v$ defined as above. Let $G(x)\in R_v[x]$ be a monic irreducible polynomial having a root
$\theta$ in its splitting field, and suppose that
\( \overline{G}(x) = \overline{g}_1(x)^{e_1} \cdots \overline{g}_t(x)^{e_t} \) is the factorization of $\overline{G}(x)$ into powers of distinct irreducible polynomials over $R_v/M_v$, where each $g_i(x)\in R_v[x]$ is monic. Then $R_v[\theta]$ is integrally closed if and only if one of the following holds:

\begin{enumerate}
\item[\textup{(i)}]
$e_i=1$ for every $1\le i\le t$;

\item[\textup{(ii)}]
There exists an index $1\le j\le t$ such that $e_j>1$.
In this case, $M_v$ is a principal ideal, say
$M_v=(\pi)$ and
$\overline{g}_j(x)$ does not divide
$\overline{Y}(x)$, where
\[ Y(x) = \frac{1}{\pi} \Bigl( G(x)-g_1(x)^{e_1}\cdots g_t(x)^{e_t} \Bigr) \in R_v[x]. \]
\end{enumerate}
\end{theorem}

In 1977, K. Uchida~\cite{Uch77} established another equivalent criterion for $R[\theta]$ to be integrally closed for a Dedekind domain $R$. This result was later generalized to Pr\"{u}fer domains by Khanduja and Jhorar~\cite{Jh-Kh}. We state their result below.

\begin{theorem}\label{valuation dede}
Let $v$, $R_v$, $M_v$, $\theta$, $G(x)$ and $g_i(x)$ be as in the notations. Then $R_v[\theta]$ is integrally closed if and only if $G(x) \notin \langle M_v, g_i(x)\rangle ^2 \  \text{in } R_v[x] \  \text{for any } i,\ 1 \le i \le t.$
\end{theorem}


\begin{remark}
Many results on monogenic extensions are established under the assumption that the underlying valuation ring is a Dedekind\footnote{An integral domain $R$ where the localization $R_v$ is a discrete valuation ring (DVR) for valuation $v$ is called a Dedekind or Noetherian Prüfer domain.}. In contrast, our results are formulated for Krull valuations of arbitrary rank, allowing the valuation ring to be non-Noetherian. The only additional hypothesis needed in our arguments is local in nature: whenever the polynomial under consideration has a repeated root modulo $M_v$ and the discriminant lies in $M_v$, we assume that the maximal ideal $M_v$ is principal. This condition is sufficient for the discriminant arguments developed later and extends the classical Dedekind domain framework to a broader class of Pr\"{u}fer domains.
\end{remark}

The following proposition gives a standard characterization of discrete valuation rings. The equivalence of conditions $(1)$, $(2)$, and $(3)$ is a classical result (see, Bourbaki~\cite[Proposition 9, p.392]{MR979760}). The equivalence of $(1)$ and $(4)$ follows from the fact that the only convex subgroups of $(\mathbb{Z},+)$ are $\mathbb{Z}$ and ${0}$. We first recall a standard characterization of discrete valuation rings, which will be useful in the sequel.

\begin{proposition}\label{prop:equivalences}
Let $v$ be a non-trivial Krull valuation of a field $F$ with valuation ring $R_v$. Then the following statements are equivalent.
\begin{enumerate}
\item $R_v$ is a discrete valuation ring.
\item $R_v$ is Noetherian.
\item $R_v$ is a principal ideal domain.
\item The valuation $v$ has rank one and its value group is isomorphic to $\mathbb Z$.
\end{enumerate}
Moreover, whenever any of these conditions hold, the maximal ideal $M_v$ is principal.
\end{proposition}

We first record a useful observation concerning a change of generator. It shows that, under a suitable condition on $k$, the integral closedness of $R_v[\theta]$ is equivalent to that of $R_v[k\theta+a]$.

\begin{proposition} \label{pro3.7}
Let $f(x) \in R_v[x]$ be a monic irreducible polynomial of degree $n$, and let $K = F(\theta)$ where $\theta$ is a root of $f(x)$. Suppose $k, a \in R_v$ are elements such that $v(k) = 0$. Then, $R_v[\theta]$ is integrally closed if and only if $R_v[k\theta + a]$ is integrally closed.
\end{proposition}
\begin{proof}
Set $\alpha=k\theta+a$. Since $k,a\in R_v$, we have $\alpha\in R_v[\theta]$. Hence $R_v[\alpha]\subseteq R_v[\theta]$. Moreover, $v(k)=0$ implies that $k$ is a unit of  $R_v$. Thus $k^{-1}\in R_v$. From $\alpha=k\theta+a$, we obtain $\theta=k^{-1}(\alpha-a)\in R_v[\alpha]$. Consequently, $R_v[\theta]\subseteq R_v[\alpha]$. Therefore, $R_v[\theta]=R_v[\alpha]=R_v[k\theta+a]$.

Hence, the two rings are identical and so $R_v[\theta]$ is integrally closed if and only if $R_v[k\theta+a]$ is integrally closed. This completes the proof.
\end{proof}

We now prove a simple ideal theoretic identity, which will be used  in the sequel.

\begin{proposition} \label{lem:ideal_split}
Let $M_v=\langle\pi\rangle$ be a maximal ideal in $R_v$ and $g(x) \in R_v[x]$ be a monic polynomial that is irreducible modulo $M_v$. Then  $\langle M_v, g(x) \rangle^2 = \langle M_v^2, g(x) \rangle \cap \langle M_v, g(x)^2 \rangle.$
\end{proposition}

\begin{proof}
The containment $\langle M_v, g(x) \rangle^2 \subseteq \langle M_v^2, g(x) \rangle \cap \langle M_v, g(x)^2 \rangle$ follows from the definition of the product of ideals. For the reverse inclusion, let $r(x)$ be a polynomial belonging to the intersection $\langle M_v^2, g(x) \rangle \cap \langle M_v, g(x)^2 \rangle$. Then there exist polynomials $a(x), b(x), c(x), d(x) \in R_v[x]$ such that
\begin{equation} \label{eq:lem_ident}
\pi^2 a(x) + g(x) b(x) = r(x) = \pi c(x) + g^2(x) d(x).
\end{equation}
By above relation, we obtain that $\pi(c(x)-\pi a(x))\in\langle g(x)\rangle$. As $\pi\not\in\langle g(x) \rangle$, implies that $c(x)\equiv\pi a(x) \pmod{g(x)}$. Write $c(x)=\pi a(x)+g(x)t(x)$ for some $t(x)\in R_v[x]$. Now substituting this value in second relation of \eqref{eq:lem_ident}, we obtain 
\[r(x)=\pi(\pi a(x)+g(x)t(x))+g^2(x)d(x).\]
Thus, $r(x) \in \langle M_v, g(x) \rangle^2$, which completes the proof.
\end{proof}

\begin{theorem}\label{Thm:4.1}
 Let $f(x)$, $h(x)$ and $g(x)$ be monic polynomials in $R_v[x]$ such that $g(x)$ is irreducible modulo $M_v=\langle\pi \rangle$. Suppose that $f(x) \in \langle \pi,\, g(x)^2 \rangle$. Then $f(x) \in \langle \pi,\,g(x) \rangle^2$ if and only if $f(h(x)) \in \langle \pi,\,g(h(x)) \rangle^2$.
\end{theorem}

\begin{proof}
By hypothesis, $f(x) \in \langle \pi,\, g(x)^2 \rangle$. Since $g(x)$ is monic, we can use the division algorithm in $R_v[x]$ to write  $f(x) = g(x)^2 q(x) + \pi r(x)$ for some $q(x), r(x) \in R_v[x]$.  
Applying the division algorithm again to divide $r(x)$ by $g(x)$, we have  
\[
    r(x) = g(x) s(x) + t(x),
\]
where $\deg t(x) < \deg g(x)$.  
Substituting this expression for $r(x)$ back into our first equation yields  
\begin{equation}\label{eq3.1}
  f(x) = g(x)^2 q(x) + \pi\, g(x) s(x) + \pi \,t(x).  
\end{equation}
Notice that the first two terms on the right-hand side belong to $\langle \pi,\, g(x) \rangle^2$. Therefore, by above equation we have 
\begin{equation} \label{eq3.3}
    f(x) \in \langle \pi,~g(x) \rangle^2 \iff \pi \, t(x) \in \langle \pi,~g(x) \rangle^2.
\end{equation}
Now, replacing $x$ with the polynomial $h(x)$ in \eqref{eq3.1}, we obtain
\begin{equation*}
    f(h(x)) = g(h(x))^2 \, q(h(x)) + \pi \, g(h(x)) \, s(h(x)) + \pi\, t(h(x)).
\end{equation*} 
By the same reasoning as for $f(x)$, 
we deduce
\begin{equation}\label{eq3.4}
    f(h(x)) \in \langle \pi,\,g(h(x)) \rangle^2 \iff \pi\,t(h(x)) \in \langle \pi,\,g(h(x)) \rangle^2.
\end{equation} 
In view of \eqref{eq3.3} and \eqref{eq3.4}, the theorem reduces to showing that
\begin{equation}\label{eq3.5}
    \pi \, t(x) \in \langle \pi,~g(x) \rangle^2 \iff \pi \, t(h(x)) \in \langle \pi,~g(h(x)) \rangle^2.
\end{equation} 
If $\pi \, t(x) \in \langle \pi,\,g(x) \rangle^2$, then substituting $h(x)$  gives $\pi\, t(h(x)) \in \langle \pi,\,g(h(x)) \rangle^2$. To prove the converse, assume that $\pi \, t(h(x)) \in \langle \pi,~g(h(x)) \rangle^2$. This means there exist polynomials, $h_1(x), h_2(x), h_3(x) \in R_v[x]$ such that
\begin{equation}\label{eqns:3.6}
    \pi\,t(h(x)) = \pi^2\, h_1(x) + \pi\,g(h(x)) \, h_2(x) + g(h(x))^2\, h_3(x).
\end{equation}
Reducing this equation modulo $M_v$, the terms multiplied by $\pi$ vanish, leaving
\[
    g(h(x))^2\, h_3(x) \equiv 0 \pmod{M_v}.
\]
Since $g(x)$ and $h(x)$ are monic, the composition $g(h(x))$ is also monic and thus non-zero modulo $M_v$. Since $(R_v/M_v)[x]$ is an integral domain, it follows that $h_3(x) \equiv 0 \pmod{M_v}$. We can therefore write $h_3(x) = \pi h_4(x)$ for some $h_4(x) \in R_v[x]$. Substituting $h_3(x) = \pi h_4(x)$ into \eqref{eqns:3.6} and dividing throughout by $\pi$, we obtain
\begin{align*}
  t(h(x)) &= \pi\, h_1(x) + g(h(x)) \, h_2(x) + g(h(x))^2\, h_4(x),\\  
  &= \pi\, h_1(x) + g(h(x)) \big[h_2(x) + g(h(x))\, h_4(x)\big].
\end{align*}
Reducing this equation modulo $M_v$ yields
\[
    t(h(x)) \equiv g(h(x)) \big[h_2(x) + g(h(x))\, h_4(x)\big] \pmod{M_v}.
\]
Suppose, for the sake of contradiction, that $h_2(x) + g(h(x))\, h_4(x) \not\equiv 0 \pmod{M_v}$ and for $j \geq 0$, let $b_jx^j$ be its leading non-zero term modulo $M_v$. Then using $g(h(x))$ is monic, right-hand side of the congruence contains a term $b_j x^{\deg(g\circ h)+j}$ which is non-zero modulo $M_v$. But we have $\deg(t) < \deg(g)$ and hence $\deg(t \circ h) < \deg(g \circ h)$, this means that every non-zero term modulo $M_v$ on left-hand side has degree less than $\deg(g \circ h)$, which is a contradiction.
Thus, we have $ t(h(x)) \equiv 0 \pmod{M_v}.$
Finally, we show that this forces $t(x) \equiv 0 \pmod{M_v}$. Let $t(x) = a_d x^d + a_{d-1} x^{d-1} + \cdots + a_0$. Evaluating at $h(x)$ gives
\[
    t(h(x)) = a_d h(x)^d + a_{d-1} h(x)^{d-1} + \cdots + a_0 \equiv 0 \pmod{M_v}.
\]
Since $h(x)$ is monic, the leading term of this expansion is exactly $a_d x^{d \cdot \deg(h)}$. For $t(h(x))$ to be the zero polynomial modulo $M_v$, its leading coefficient must be zero, so $a_d \equiv 0 \pmod{M_v}$. Continuing this process inductively for the remaining coefficients, we see that every coefficient of $t(x)$ must be zero modulo $M_v$. 
Thus, $t(x) \equiv 0 \pmod{M_v}$, which means $t(x) = \pi t_1(x)$ for some polynomial $t_1(x) \in R_v[x]$. Multiplying by $\pi$, we obtain $\pi t(x) = \pi^2 t_1(x)$, .i.e., $\pi t(x)$ lies in the ideal $\langle \pi \rangle^2 \subseteq \langle \pi, g(x) \rangle^2$. This proves the reverse implication of \eqref{eq3.5} which completes the proof of the theorem.
\end{proof}

\begin{corollary}\label{Cor:4.2}
 Let $f(x)$ and $g(x)$ be monic polynomials in $R_v[x]$ such that $g(x)$ is irreducible modulo $M_v=\langle\pi \rangle$. Suppose that $f(x) \in \langle \pi,\, g(x)^2 \rangle$ and $p$ is a prime such that $v(p)>0$. Then $f(x) \in \langle \pi,\,g(x) \rangle^2 $ if and only if $f(x^p) \in \langle \pi,\,g(x) \rangle^2.$
\end{corollary}
\begin{proof}
    Applying Theorem \ref{Thm:4.1} with $h(x)=x^p$ to the polynomial $f(x)$, we have
    \begin{equation*}
        f(x) \in \langle \pi,\,g(x) \rangle^2 \iff f(x^p) \in \langle \pi,\,g(x^p) \rangle^2.
    \end{equation*}
    We wish to point out that the irreducibiliy of $g(x)$ was only required for the above equivalence. The statement proved below follows for any polynomial $g(x)$. To prove the corollary, we only need to show that $f(x^p) \in \langle \pi,\,g(x^p) \rangle^2 \iff f(x^p) \in \langle \pi,\,g(x) \rangle^2$. Let $t(x)$ be the polynomial such that $g(x^p) = g(x)^p + p t(x)$.

    \noindent Now assume, $f(x^p) \in \langle \pi,\,g(x^p) \rangle^2 $, i.e., there exists polynomials $h_1(x), h_2(x), h_3(x) \in \Z[x]$ such that
    \begin{align*}
        f(x^p) &= \pi^2h_1(x) + \pi g(x^p) h_2(x) + g(x^p)^2 h_3(x), \\
        &= \pi^2 h_1(x) + \pi g(x)^ph_2(x) + \pi p t(x) + \big[ g(x)^{2p} + p g(x)^p \times 2 + p^2 t(x)^2 \big] h_3(x), \\
        &= \pi^2 h_4(x) + \pi g(x) h_5(x) + g(x)^2 h_6(x)
    \end{align*}
    for some $h_4(x), h_5(x), h_6(x) \in \Z[x]$, where the last relation is clubbed using the fact $v(p)>0$. This implies $f(x^p) \in \langle \pi,\,g(x) \rangle^2$. Using similar manipulations, we obtain that other implication. This completes the proof of the corollary.
\end{proof}



\begin{lemma} \label{lem:5.14}
Let $v$ be a Krull valuation on a field $F$ with valuation ring $R_v$ and its maximal ideal is principal, say $M_v=(\pi)$. Suppose that $\overline{g}(y)$ is an irreducible repeated factor of $\overline{f}(y)$ in $k_v[y]$ and  $\overline{g}_1(x)$ be an irreducible factor of $\overline{g}(\overline{h}(x))$. If \( f(h(x)) \notin \langle \pi,g(h(x)) \rangle^2, \) then \( f(h(x)) \notin \langle \pi,g_1(x) \rangle^2. \)
\end{lemma}
\begin{proof}
We prove the contrapositive. Assume that
$f(h(x))\in\langle\pi,g_1(x)\rangle^2$.
Since $\overline g(y)$ is a repeated factor of $\overline f(y)$, there exists $\overline Q(y)\in k_v[y]$ such that
$\overline f(y)=\overline g(y)^2\overline Q(y)$.
Choosing a lift $Q(y)\in R_v[y]$ of $\overline Q(y)$, we have
$f(y)-g(y)^2Q(y) \in M_v = \langle \pi \rangle$ and hence there exists $A(y)\in R_v[y]$ such that
\begin{equation} \label{eq:lifting}
f(y) = g(y)^2 Q(y) + \pi A(y).
\end{equation}
Substituting $y=h(x)$ into \eqref{eq:lifting} gives
\begin{equation} \label{eq:composition}
f(h(x)) = g(h(x))^2Q(h(x)) + \pi A(h(x)).
\end{equation}
Since $\overline g_1(x)\mid\overline g(\overline h(x))$, we have
$g(h(x))\in\langle\pi,g_1(x)\rangle$ and therefore
$g(h(x))^2\in\langle\pi,g_1(x)\rangle^2$. Hence \eqref{eq:composition} implies
$\pi A(h(x)) \in \langle \pi,g_1(x) \rangle^2$. Thus, there exists polynomials $U(x), V(x), W(x) \in R_v[x]$ satisfying the relation
\[
\pi A(h(x)) =
\pi^2U(x) +\pi g_1(x)V(x) + g_1(x)^2W(x).
\]
Reducing modulo $\pi$, the above equation yields $\overline g_1(x)^2\overline W(x)=0$. Since $k_v[x]$ is an integral domain and $\overline g_1(x) \neq 0$, we obtain $\overline W(x)=0$. So, $W(x)=\pi\widetilde W(x)$ for some $\widetilde W(x)\in R_v[x]$. Cancelling the factor $\pi$ yields
\[
A(h(x)) = \pi U(x)+g_1(x)V(x)+ g_1(x)^2 \widetilde W(x).
\]
This implies that
$A(h(x))\in\langle\pi,g_1(x)\rangle$. Let $\alpha$ be a root of $\overline g_1(x)$ in an algebraic closure of $k_v$ and put $\beta=\overline h(\alpha)$. Since $\overline g_1(x)\mid\overline A(\overline h(x))$, we have $\overline A(\beta)=0$. Moreover, $\overline g_1(x)\mid\overline g(\overline h(x))$ implies $\overline g(\beta)=0$. As $\overline g(y)$ is irreducible over $k_v$, it is the minimal polynomial of $\beta$ and hence $\overline g(y)\mid\overline A(y)$. Therefore,
$A(y)=g(y)B(y)+\pi C(y)$ for some $B(y),C(y)\in R_v[y]$.
Substituting $y=h(x)$ into this identity and then into \eqref{eq:composition}, we obtain
\begin{align}
f(h(x)) &= g(h(x))^2Q(h(x)) + \pi g(h(x))B(h(x)) + \pi^2C(h(x)),
\end{align}
which shows that
$f(h(x))\in\langle\pi,g(h(x))\rangle^2$.
This proves the contrapositive.
\end{proof}
\noindent \textbf{Remark:} The significance of Lemma \ref{lem:5.14} is that, when applying Dedekind's criterion to the polynomial $f(h(x))$, one does not need to factor $g(h(x))$ further. More precisely, suppose that if every irreducible factor $g(x)$ of $\overline{f}(x)$ is a repeated factor of $\overline{f}(x)$. Then, in order to ensure that
$R_v[\theta]$ is integrally closed, where $\theta$ is a root of the irreducible polynomial $f(h(x))$, it is enough to verify that 
\( f(h(x))\notin \langle \pi,g(h(x))\rangle^2 \)
for every irreducible factor $g(x)$ of $\overline{f}(x)$. Indeed, Lemma \ref{lem:5.14} shows that this condition automatically implies 
\( f(h(x))\notin \langle \pi,g_1(x)\rangle^2 \)
for every irreducible factor $g_1(x)$ of $\overline{g(h(x))}$. Thus, no further factorization of $g(h(x))$ is required in the application of Dedekind's criterion.

\begin{lemma} \label{lemma:3.6}
 Assume that the residue field $k_v$ has prime characteristic $p > 0$ and $M_v = \langle \pi \rangle$. Let $f(x)$, $h(x)$ and $g(x)$ be monic polynomials in $R_v[x]$ such that $g(x)$ is irreducible modulo $M_v$. If $f(x) \in \langle \pi,\, g(x)^2 \rangle$, then
$f(x^p) \in \langle \pi,\,g(x) \rangle^2 \,\, \text{if and only if} \,\, f(h(x^p)) \in \langle \pi,\,g(h(x)) \rangle^2.$
\end{lemma}

\begin{proof}
 Using the assumption $f(x) \in \langle \pi,\, g(x)^2 \rangle$ by Corollary~\ref{Cor:4.2}, we have $f(x^p) \in \langle \pi,\, g(x) \rangle^2$ if and only if $f(x) \in \langle \pi,\, g(x) \rangle^2$.
By Theorem~\ref{Thm:4.1}, this latter condition holds if and only if $f(h(x)) \in \langle \pi,\, g(h(x)) \rangle^2$. To prove the result, it only remains to show that $f(h(x)) \in \langle \pi,\, g(h(x)) \rangle^2$ if and only if $f(h(x^p)) \in \langle \pi,\, g(h(x)) \rangle^2$. In view of $v(p)>0$ and the arguments analogues to those used in Corollary \ref{Cor:4.2}, proving the above equivalence is same as proving $f(h(x)) \in \langle \pi,\, g(h(x)) \rangle^2$ if and only if $f(h(x^p)) \in \langle \pi,\, g(h(x^p)) \rangle^2$.

If $f(h(x)) \in \langle \pi, g(h(x)) \rangle^2$ then trivially we have $f(h(x^p)) \in \langle \pi, g(h(x^p)) \rangle^2$. Now suppose $f(h(x^p)) \in \langle \pi, g(h(x^p)) \rangle^2$ and let $g_1(x)$ be the irreducible factor of $g(h(x))$ modulo $M_v$. Then we have $(f\circ h)(x^p) \in \langle \pi, g_1(x^p) \rangle^2$. As $g_1(x)$ is an irreducible repeated factor of $f\circ h$ modulo $M_v$, Theorem \ref{Thm:4.1} with $f = f \circ h$ and $h(x) = x^p$ yields $(f \circ h)(x) \in \langle \pi, g_1(x) \rangle^2$. Now since $g(x)$ is an irreducible repeated factor of $f(x)$ and $g_1(x)$ is an irreducible factor of $g \circ h$, applying Lemma \ref{lem:5.14} yields $f(h(x)) \in \langle \pi, g(h(x)) \rangle^2$. This completes the proof of the lemma.
\end{proof}

\begin{proposition} \label{pro:5.13}
Let $M_v=\langle\pi \rangle$ be maximal ideal and  $f(x), h(x) \in R_v[x]$ be monic polynomials such that the composition $\mathcal{F}(x) = f(h(x))$ is irreducible  over $F$. Let $\alpha$ be a root of $f(x)$ and $\theta$ be a root of $\mathcal{F}(x)$. If $R_v[\alpha]$ is not integrally closed, then $R_v[\theta]$ is also not integrally closed.
\end{proposition}

\begin{proof}
If $R_v[\alpha]$ is not integrally closed, then by Theorem~\ref{valuation dede}, there exists a monic polynomial $g(x) \in R_v[x]$ such that its reduction $\overline{g}(x)$ is an irreducible repeated factor of $f(x)$ modulo $M_v$. This implies that $f(x) \in \langle \pi, g(x) \rangle^2.$ By substituting $h(x)$ for $x$ in this ideal relation, it follows that
\begin{equation}\label{eqn;513}
     \mathcal{F}(x) \in \langle M_v, g(h(x)) \rangle^2.
\end{equation}
  Let $\overline{g_1}(x)$ be an irreducible factor of $\overline{g}(\overline{h}(x))$ in the residue field $k_v$, and let $g_1(x) \in R_v[x]$ be a monic lift of this factor. As $g(x)$ was a repeated factor of $f(x)$ modulo $M_v$, it is clear that $g(h(x))$ is a repeated factor of $\cF(x)$ modulo $M_v$. Thus, we have $g_1(x)$ is a repeated factor of $\cF(x)$ modulo $M_v$. Now using $\overline{g_1}(x)$ divides $\overline{g}(\overline{h}(x))$, we have
\[
    \langle M_v, g(h(x)) \rangle \subseteq \langle M_v, g_1(x) \rangle.
\]
The square of both sides of above inclusion yields $\langle M_v, g(h(x)) \rangle^2 \subseteq \langle M_v, g_1(x) \rangle^2$. By \eqref{eqn;513} and the above inclusion, we have  $\mathcal{F}(x)\in \langle M_v, g_1(x) \rangle^2$. Thus, we have proved that $\cF(x)$ has a repeated factor $g_1(x)$ modulo $M_v$ such that $\mathcal{F}(x)\in \langle M_v, g_1(x) \rangle^2$. Applying Theorem~\ref{valuation dede}, we conclude that $R_v[\theta]$ is not integrally closed. This completes the proof.
\end{proof}

\begin{proposition}\label{pro:5.7}
Let $f(x),\, h(x)\in R_v[x]$, where $h(x)=x^2q(x)+ax+b$, such that
$\mathcal{F}(x)=f(h(x))$ is monic and irreducible over $F$.
Let $\theta$ be a root of $\mathcal{F}(x)$. If $a f'(b) \in M_v$ and 
$f(b) \in M_v^2$, then $R_v[\theta]$ is not integrally closed.
\end{proposition}
\begin{proof}
We have $\mathcal{F}(0)=f(b)$. Since $f(b) \in M_v^2$, it follows that $x$ is a repeated factor of $\mathcal{F}(x)$ modulo $M_v$. Let $g(x)=x$ be a monic lift of this factor in $R_v[x]$. Then
$\mathcal{F}(x) \in \langle M_v,g(x)\rangle^2$. By Theorem~\ref{valuation dede}, it follows that $R_v[\theta]$ is not integrally closed. This completes the proof.
\end{proof}

\begin{lemma}\label{pro:5.12}
 Let $f(x) \in R_v[x]$ be a monic polynomial of degree $n$ and  $h(x) = x^m + a x^{m-1} + b \in R_v[x]$. Suppose that $f(h(x))$ is irreducible over $F$ and $v(m)=0$. Let $\theta$ be a root of $f(h(x))$, and set $L = F(\theta)$. Furthermore, assume that 
\[
m^{mn} f\left(h\left(\frac{-a(m-1)}{m}\right)\right) \in M_v^2.
\]
Then the simple ring extension $R_v[\theta]$ is not integrally closed.
\end{lemma}

\begin{proof}
Define $z := m\theta + a(m-1)$. Since $\theta$ is a root of $f(h(x))$, the minimal polynomial of $z$ over $F$ is given by
\[
    H(x) = m^{mn} f\left(h\left(\frac{x - a(m-1)}{m}\right)\right).
\]
The constant term of $H(x)$ is
\[
    H(0) = m^{mn} f\left(h\left(\frac{-a(m-1)}{m}\right)\right).
\]
By our hypothesis, $H(0) \in M_v^2$. Next, we compute the derivative of $H(x)$ at $x = 0$ to determine the multiplicity of $x$ modulo $M_v$. By the chain rule, we have
\[
    H'(x) = m^{mn-1} f'\left(h\left(\frac{x - a(m-1)}{m}\right)\right) h'\left(\frac{x - a(m-1)}{m}\right).
\]
Using $h'(x) = x^{m-2}(mx + a(m-1))$, we have
\[
    h'\left(\frac{-a(m-1)}{m}\right) =  0.
\]
Consequently, $H'(0) = 0$. Since $\overline{H}(0) = \overline{0}$ and $\overline{H}'(0) = \overline{0}$ in the residue field $R_v/M_v$, $x$ is a repeated irreducible factor of $\overline{H}(x)$. Therefore, in $R_v[x]$, $H(x)$ can be expanded as $x^2 Q(x) + c_1 x + H(0)$ for some $Q(x) \in R_v[x]$, where the linear coefficient is $c_1 = H'(0) = 0$. Since the linear term vanishes and the constant term $H(0) \in M_v^2$, every term of $H(x)$ belongs to the square of the ideal $\langle M_v, x \rangle$. That is, $H(x) \in \langle M_v, x \rangle^2$. By Theorem~\ref{valuation dede}, the condition $H(x) \in \langle M_v, x \rangle^2$ implies that the extension $R_v[z] = R_v[m\theta + a(m-1)]$ is not integrally closed. Now using $v(m)=0$ and by Proposition~\ref{pro3.7}, we obtain that $R_v[\theta]$ is not integrally closed. This completes the proof.
\end{proof}

The following well known lemma will be used in proof of our main result. Interested readers can see the standard book by Narkiewicz \cite[Proposition~4.15]{Nark} for the proof.

\begin{lemma}\label{lemma:3.10}
Let $K$ and $L$ be a number field with $K \subset L$. Then $\D_K^{\,[L:K]} \mid \D_L.$
\end{lemma}

\begin{proposition}\label{Prop:4.13}
    Suppose $M_v$ is a principal ideal generated by $\pi$ and the residue field $k_v$ is perfect. Let $f(x) \in R_v[x]$ be a monic polynomial of degree $n$ and  $h(x) = x^{m} + a x^{m-1} + b \in R_v[x]$. Suppose that $\min(v(a), v(m)) = 0$ and that the composition $\mathcal{F}(x)=f(h(x))$ is irreducible and separable over $F$. Let $\theta$ and $\alpha$ be roots of $\mathcal{F}(x)$ and $f(x)$, respectively. Then  $R_v[\theta]$ is not integrally closed if and only if at least one of the following holds:
\begin{enumerate}
    \item  $R_v[\alpha]$ is not integrally closed;
    \item $f(b) \in M_v^2$ whenever $m \geq 3$ or $af'(b) \in M_v$;
    \item $m^{mn} f\!\left(h\!\left(\dfrac{-a(m-1)}{m}\right)\right) \in M_v^2$ whenever $v(m)=0$.
\end{enumerate}
\end{proposition}
\begin{proof}
Assuming that one of the conditions \textup{(1)}, \textup{(2)} or \textup{(3)} are satisfied, it follows from Propositions~\ref{pro:5.13}, \ref{pro:5.7}, and Lemma \ref{pro:5.12}, respectively, that the simple ring extension $R_v[\theta]$ is not integrally closed.\\[0.5mm]
We prove the converse by contraposition. Suppose that $R_v[\alpha]$ is integrally closed and that $f(b)\notin M_v^2$ whenever $a \in M_v$. If $v(m)=0$, we further assume that
\begin{equation} \label{eqn:4.11}
    m^{mn}f\!\left(h\!\left(\frac{-a(m-1)}{m}\right)\right)\notin M_v^2.
\end{equation}
First, we wish to point that under the assumption $\min\{ v(a),v(m) \} = 0$, \eqref{eqn:4.11} holds trivially in the case when $v(m) \neq 0$. Suppose $v(m) \neq 0$, we use the assumption $\min\{ v(a),v(m) \} = 0 $ to get $v(a)=0$. Let $f(x)=x^n+\sum_{i=0}^{n-1}a_ix^i$. Replace $x$ by $h\left(\frac{-a(m-1)}{m}\right)$, we obtain
\[\mathcal{F}\left(\frac{-a(m-1)}{m}\right)=\left((-1)^{m-1}\left(\frac{a}{m}\right)^m(m-1)^{m-1}+b\right)^n+\sum_{i=0}^{n-1}a_i\left((-1)^{m-1}\left(\frac{a}{m}\right)^m(m-1)^{m-1}+b\right)^i\]
Multiplying both side by $m^{mn}$ in above equation yields
\[m^{mn}\mathcal{F}\left(\frac{-a(m-1)}{m}\right)=[(-1)^{m-1}a^m(m-1)^{m-1}+bm^m]^n+m^m\sum_{i=0}^{n-1}m^{m(n-i-1)}a_i[(-1)^{m-1}a^m(m-1)^{m-1}+bm^m]^i\]
Taking valuation of both sides in the above equation
\[v\left(m^{mn}\mathcal{F}\left(\frac{-a(m-1)}{m}\right)\right)=v\left([(-1)^{m-1}a^m(m-1)^{m-1}+bm^m]^n+m^mA\right),\]
where $A=\sum_{i=0}^{n-1}m^{m(n-i-1)}a_i[(-1)^{m-1}a^m(m-1)^{m-1}+bm^m]^i.$ Now using definition of krull valuation and keeping in mind that $v(m)\neq0$, we obtain 
\[v \left( m^{mn}f \left( h \left( \frac{-a(m-1)}{m} \right) \right) \right)= v(a^m (m-1)^{m-1}) = 0 \]
Thus we have shown that \eqref{eqn:4.11} holds irrespective of the assumption on $m$ in \textup{(3)}.
We prove the result for two different cases according to the condition $f(b) \notin M_v$ or $f(b) \in M_v$.\\[0.5mm]
\noindent\textbf{Case I.} First assume $f(b) \notin M_v$, which implies $v(f(b)) = 0$. Let $K = F(\alpha)$ be the field generated by a root of $f(x)$ and  $\mathfrak{D}_{K/F}$ be its field discriminant. Since $R_v[\alpha]$ is integrally closed, the valuation of the polynomial discriminant matches the field discriminant:
\[
    v(\D_f) = v(\mathfrak{D}_{K/F}).
\]
Let $L = F(\theta)$ be the field generated by a root of $\mathcal{F}(x)$ and $\bA_L$ denote the algebraic closure of $R_v$ in $L$. Clearly, the index $[L:K]= \deg h = m$. Applying Lemma \ref{lemma:3.10}, we have 
\begin{equation} \label{eqn:4.12}
    v(\D_{L/F}) \geq m v(\D_{K/F})
\end{equation}
From Theorem~\ref{Discriminant}, recall that the polynomial discriminants $\D_{\cF}$ and $\D_f$ are related by the formula:
\[
    \D_{\mathcal{F}} = \pm \, \D_f^{m} \, f(b)^{m - 2}\,m^{nm}\,  f\left(h\left(\frac{-a (m-1)}{m}\right)\right).
\]
Using the above formula along with \eqref{eqn:4.12}, we get
\begin{equation*} 
    v(\D_{\mathcal{F}}) - v(\mathfrak{D}_{L/F}) \le m v(\D_f) + (m-2)v(f(b)) + v\left(m^{nm}f\left(h\left(\frac{-a(m-1)}{m} \right) \right) \right) - m v(\mathfrak{D}_{K/F}) .
\end{equation*}
By assumption $v(f(b)) = 0$ and the observation $v(\D_f) = v(\mathfrak{D}_{K/F})$, the above inequality simplifies to
\[
    v(\D_{\mathcal{F}}) - v(\mathfrak{D}_{L/F}) \le v\left(m^{nm}f\left(h\left(\frac{-a(m-1)}{m}\right)\right)\right).
\]
Since $m^{nm} f\!\left(h\!\left(\frac{-a(m-1)}{m}\right)\right) \notin M_v^2$, using the standard relation \eqref{formula} we obtain
\[
   2\,v\!\left([\bA_L:R_v[\theta]]\right)= v(\D_{\mathcal{F}}) - v(\mathfrak{D}_{L/F}) \le v\left(m^{nm}f\left(h\left(\frac{-a(m-1)}{m}\right)\right)\right) \leq 1.
\]
Since the module index $[\bA_L : R_v[\theta]]$ is always an integer, the above holds only when $v([\bA_L : R_v[\theta]]) = 0$.
This implies that  $R_v[\theta]$ is integrally closed.

\noindent\textbf{Case II.} Assume $f(b) \in M_v$. By the hypotheses of the converse, we  know that $f(b) \notin M_v^2$ and $$m^{mn} f\!\left(h\!\left(\frac{-a(m-1)}{m}\right)\right) \notin M_v^2.$$ On the contrary assume that  $R_v[\theta]$ is not integrally closed. Then, by Theorem~\ref{valuation dede}, there exists a monic polynomial $g_1(x) \in R_v[x]$, which is irreducible repeated factor of $\cF$ modulo $M_v$ and $\mathcal{F}(x) \in \langle M_v, g_1(x) \rangle^2$. 
Let $\gamma$ be a root of $g_1(x)$ in the algebraic closure of the residue field $k_v$. Then $\gamma$ is a repeated root of $\mathcal{F}(x)$ modulo $M_v$. Thus we have $\mathcal{F}(\gamma) \equiv 0 \pmod{M_v}$ and the derivative $\mathcal{F}'(\gamma) \equiv 0 \pmod{M_v}$. Using the chain rule, the derivative of the composition is
\[
    \mathcal{F}'(x) = f'(h(x)) h'(x) = f'(h(x)) x^{m-2} \bigl(mx + a(m-1)\bigr).
\]
Evaluating this at $x = \gamma$, we obtain the congruence:
\[
    f'(h(\gamma)) \, \gamma^{m-2} \bigl(m\gamma + a(m-1)\bigr) \equiv 0 \pmod{M_v}.
\]
Since  $M_v$ is prime ideal, which implies at least one of the factors must be zero modulo $M_v$. Accordingly, we have three possible subcases to look at\\[1mm]
\noindent \textit{\textbf{Subcase 2.1}: $\gamma \equiv 0 \pmod{M_v}$.} If $\gamma \equiv 0 \pmod{M_v}$, then its minimal polynomial modulo $M_v$ is $g_1(x) \equiv x \pmod{M_v}$. A simple expansion of $f(h(x))$ will yield that $x$ is a repeated factor of $\cF(x)$ modulo $M_v$ if and only if either $m \geq 3$ or $m=2$ and $af'(b) \in M_v$.
The condition $\mathcal{F}(x) \in \langle M_v, x \rangle^2$ implies that the constant term of $\mathcal{F}(x)$ must belong to $M_v^2$. However, the constant term is $\mathcal{F}(0) = f(b)\notin M_v^2$ under these assumptions, which yields a contradiction. Thus, $\gamma \not\equiv 0 \pmod{M_v}$.\\[1mm]
\noindent \textit{\textbf{Subcase 2.2}: $f'(h(\gamma)) \equiv 0 \pmod{M_v}$.} We already know $f(h(\gamma)) = \mathcal{F}(\gamma) \equiv 0 \pmod{M_v}$. If we also have $f'(h(\gamma)) \equiv 0 \pmod{M_v}$ as well, then $h(\gamma)$ is a repeated root of $f(x)$ modulo $M_v$. Since, $k_v$ is perfect, there exists a monic polynomial $g(x) \in R_v[x]$ irreducible modulo $M_v$ such that $g(x)^2$ divides $f(x)$ modulo $M_v$ and $g(h(\gamma)) \equiv 0 \pmod{M_v}$. 
Since $g_1(\gamma) \equiv 0 \pmod{M_v}$ and $g_1(x)$ is irreducible modulo $M_v$, it follows that $g_1(x)$ divides $g(h(x))$ modulo $M_v$. Since $g(x)^2$ divides $f(x)$ modulo $M_v$, we have $f(x) \in \langle M_v, g(x)^2 \rangle$. However, by our assumption that the extension $R_v[\alpha]$ is integrally closed, we have $f(x) \notin \langle M_v, g(x) \rangle^2$. Combining these two facts and applying Theorem~\ref{Thm:4.1}, it follows that $f(h(x)) \notin \langle M_v, g(h(x)) \rangle^2$. Using these observations along with Lemma~\ref{lem:5.14}, we deduce that $\cF(x) = f(h(x)) \notin \langle M_v, g_1(x) \rangle^2$. This contradicts our initial assumption for Case II. Thus, $f'(h(\gamma)) \not\equiv 0 \pmod{M_v}$.\\[0.5mm]
\noindent \textit{\textbf{Subcase 2.3}: $m\gamma + a(m-1) \equiv 0 \pmod{M_v}$.} If $v(m)\neq 0$, then $m\in M_v=\langle\pi\rangle$. Hence, $a(m-1)\in M_v$. Since $M_v$ is a prime ideal, it follows that either $a\in M_v$ or $m-1\in M_v$. On the other hand
$\min\{v(a),v(m)\}=0$ and $v(m)\neq 0$, which implies that
$v(a)=0$. Thus, $a\notin M_v$ and consequently, $m-1\in M_v$, which is a contradiction. Therefore, $v(m)$ must be zero in this subcase and hence $m$ is a unit in  $R_v$. In particular, $m$ is invertible modulo $M_v$ and we may define an element
\[
\eta=\frac{-a(m-1)}{m}\in R_v.
\] 
Thus, $\gamma \equiv \eta \pmod{M_v}$. Hence, the minimal polynomial of $\gamma$ modulo $M_v$ is $g_1(x)\equiv x-\eta \pmod{M_v}$. We now need to determine whether $\mathcal{F}(x)\in
\langle M_v,x-\eta\rangle^2$. Since $x-\eta$ is a repeated factor of $\mathcal{F}(x)$ modulo $M_v$, we already have $\mathcal{F}(x)\in \langle M_v,(x-\eta)^2\rangle$. Therefore, by
Proposition~$\ref{lem:ideal_split}$, $\mathcal{F}(x)\in
\langle M_v,x-\eta\rangle^2$ if and only if $\mathcal{F}(x)\in
\langle M_v^2,x-\eta\rangle$. Using polynomial division in $R_v[x]$, we can write
$\mathcal{F}(x)=(x-\eta)Q(x)+\mathcal{F}(\eta)$. Consequently,
$\mathcal{F}(x)\in\langle M_v^2,x-\eta\rangle$ if and only if
$\mathcal{F}(\eta)\in M_v^2$. Note that $\mathcal{F}(\eta)=f(h(\eta))$.
By our hypothesis, $m^{mn}f(h(\eta))\notin M_v^2$. Since $m$ is a unit in $R_v$, $m^{mn}$ is also a unit. Thus, the above condition is equivalent to $f(h(\eta))\notin M_v^2$ yielding $\mathcal{F}(\eta)\notin M_v^2$. Consequently, $\mathcal{F}(x)\notin\langle M_v,x-\eta\rangle^2 =\langle M_v,g_1(x)\rangle^2$, contradicting our assumption that $\mathcal{F}(x)\in\langle M_v,g_1(x)\rangle^2$.\\[1mm]
Since all three subcases lead to a contradiction, no such factor $g_1(x)$ can exist. Therefore, $\mathcal{F}(x) \notin \langle M_v, g(x) \rangle^2$ for any repeated factor $g(x)$ modulo $M_v$, and we conclude that $R_v[\theta]$ is integrally closed. This completes the proof.
\end{proof}

\begin{proposition}\label{Pro:4.14}
    Suppose  $M_v$ is a principal ideal generated by $\pi$ and  the residue field $k_v$ is a perfect field of prime characteristic $p > 0$. Let $f(x) \in R_v[x]$ be a monic polynomial of degree $n$,  $h(x) = x^m + ax^{m-1} + b \in R_v[x]$ be a monic polynomial such that $a,m \in M_v$ and the composition $\mathcal{F}(x)=f(h(x))$ is irreducible and separable over $F$. Let $\theta$ be a root of $\mathcal{F}(x)$ and $u \geq 1$ be such that $m = p^u \ell$ where $p \nmid \ell$. Let $b' \in R_v$ be such that $(b')^{p^u} \equiv b \pmod{M_v}$. Then the simple ring extension $R_v[\theta]$ is integrally closed if and only if the polynomials 
    \[ f(x^{\ell} + b') \text{ and } \frac{1}{\pi} \left( f(x^m + b) + ax^{m-1} f'(x^m + b) - f(x^{\ell} + b')^{p^u} \right) \text{ are coprime modulo } M_v. \]
    Further, if we assume that $a\in M_v^2$. Then the simple ring extension $R_v[\theta]$ is integrally closed if and only if $R_v[\eta]$ is integrally closed, where $\eta$ is a root of $f(x^p+b)$
\end{proposition}

\begin{proof}
Define the polynomial $\tilde{f}(x) = f(x^p+b)$. Suppose the reduction of $f(x)$ modulo $M_v$ factors as 
$$ f(x) \equiv \prod_{i=1}^t g_i(x)^{e_i} \pmod{M_v}, $$
where $g_i(x)$'s are irreducible modulo $M_v$ and $e_i\geq 1$.
Since the residue field is perfect and of characteristic $p$, there exists an element $b'' \in R_v$ such that $(b'')^p \equiv b \pmod{M_v}$. This allows us to write $x^p + b \equiv (x+b'')^p \pmod{M_v}$. Therefore, $\tilde{f}(x)$ can be factored modulo $M_v$ as
$$ \tilde{f}(x) \equiv \left( \prod_{i=1}^t g_i(x+b'')^{e_i} \right)^p \pmod{M_v}. $$
Let $\tilde{g}_i(x) = g_i(x+b'')$. Clearly $\tilde{g_i}(x)$ are irreducible modulo $M_v$ for all $1 \leq i \leq t$. Since this factorization is a $p$-th power, $\tilde{f}(x)$ inherently belongs to the ideal $\langle \pi, \tilde{g}_i(x)^2 \rangle$ for all $i$. For a fixed $i$ in the range $1 \leq i \leq t$, applying Theorem \ref{Thm:4.1} with the substitution $x \mapsto x^{\frac{m}{p}}$ (which is a valid polynomial substitution since $p$ divides $m$), we obtain that $\tilde{f}(x) \in \langle \pi, \tilde{g}_i(x) \rangle^2$ if and only if $\tilde{f}(x^{\frac{m}{p}}) \in \langle \pi, \tilde{g}_i(x^{\frac{m}{p}}) \rangle^2$. Expanding these compositions, this establishes that $f(x^p+b) \in \langle \pi, g_i(x+b'') \rangle^2$ if and only if $f(x^m+b) \in \langle \pi, g_i(x^{\frac{m}{p}}+b'') \rangle^2$.
Since the $f(x^m+b)$ can be written as $\left( \prod_{i=1}^t g_i(x^{\frac{m}{p}}+b'')^{e_i} \right)^p$ modulo $M_v$, we can apply Lemma \ref{lem:5.14} to $f(x^m+b)$ and then apply Theorem~\ref{valuation dede} (GDC) to both $f(x^p+b)$ and $f(x^m+b)$ to conclude that
\begin{equation} \label{eqn:3.16}
    R_v[\eta] \text{ is integrally closed if and only if } R_v[\beta] \text{ is integrally closed,}
\end{equation}
where $\beta$ is a root of $f(x^m+b)$. Next, using  Taylor expansion and the fact that $a \in M_v$, we can express the composition $\mathcal{F}(x)$ as:
\begin{equation} \label{eqn:3.17}
    \mathcal{F}(x) = f(x^m+ax^{m-1}+b) = f(x^m+b) + a x^{m-1} f^{\prime}(x^m+b) + \pi^2 H(x)
\end{equation}
for some polynomial $H(x) \in R_v[x]$. The term $\pi^2 H(x)$  arises because all other terms in the expansion involve $a^2$ and since $a \in M_v$, it follows that $a^2 \in M_v^2$. Using the factorization of $f(x)$ modulo $M_v$, we obtain
$$ \mathcal{F}(x) \equiv f(x^m+b) \equiv \left( \prod_{i=1}^{t} g_i(x^{\ell}+b')^{e_i} \right)^{p^u} \pmod{M_v}.$$
Applying the Generalized Dedekind Criterion and Lemma \ref{lem:5.14}, the extension $R_v[\theta]$ is integrally closed if and only if $\mathcal{F}(x) \notin \langle \pi, g_i(x^{\ell}+b') \rangle^2$ for all $i \in \{1,2,\dots,t\}$. In view of equation \eqref{eqn:3.17}, 
this holds if and only if 
\[
    f(x^m+b) + ax^{m-1} f^{\prime}(x^m+b) \notin \langle \pi, g_i(x^{\ell}+b') \rangle^2 \quad \text{for all } i \in \{1,2,\dots,t\}.
\]
Combining the above observation along with the fact that $f(x^{\ell}+b') \equiv \prod_{i=1}^t g_i(x^{\ell}+b')^{e_i} \pmod{M_v^2} $, we obtain that $R_v[\theta]$ is integrally closed if and only if the polynomials $$\frac{1}{\pi} \left( f(x^m + b) + ax^{m-1} f'(x^m + b) - f(x^{\ell} + b')^{p^u} \right)$$ and $f(x^{\ell} + b')$are coprime modulo $M_v$. 

Further, if we assume $a \in M_v^2$. Then the linear term $ax^{m-1} f^{\prime}(x^m+b)$ completely lies in the ideal $M_v^2$. Consequently, this term  belongs to the ideal $\langle \pi, g_i(x^{\ell}+b') \rangle^2$. The condition then simplifies to $R_v[\theta]$ is integrally closed if and only if $f(x^m+b) \notin \langle \pi, g_i(x^{\ell}+b') \rangle^2$ for all $i \in \{1,2,\dots,t\}$.

Applying the  Generalized Dedekind Criterion and Lemma \ref{lem:5.14} to this simplified condition, we see that $R_v[\theta]$ is integrally closed if and only if the extension $R_v[\beta]$ is integrally closed. Using this with the intermediate equivalence established in \eqref{eqn:3.16}, we finally deduce that $R_v[\theta]$ is integrally closed if and only if $R_v[\eta]$ is integrally closed whenever $a \in M_v^2$. This completes the proof.
\end{proof}

Our next theorem establishes criterion for monogenity of composition over any ring, which when restricted to the field $\Q$ establishes Theorem \ref{thm:2.1}. For the sake of completeness of the below mentioned theorem, we restate all the notations and terminology once again in the statement of the theorem.

\begin{theorem}\label{VT2.1}
Let $v$ be a Krull valuation of a field $F$ with valuation ring $R_v$ and maximal ideal $M_v = \langle \pi \rangle$. Assume that the residue field $R_v/M_v$ is perfect and has prime characteristic $p > 0$. Let $f(x) \in R_v[x]$ be a monic polynomial of degree $n$ and  $h(x) = x^m + a x^{m-1} + b \in R_v[x]$ be a monic polynomial. Let $t \ge 1$ be an integer. Suppose the composition $\mathfrak{F}(x)=f(h(x^t))$ is irreducible and separable over $F$ along with the assumption $v(\D_{\fF}) > 0$. Let $\Theta$ be a root of $\fF(x)$ and  $\alpha$ be a root of $f(x)$. Then the simple ring extension $R_v[\Theta]$ is integrally closed if and only if all of the following conditions hold:
\begin{enumerate}
    \item[\textup{(1)}] $R_v[\alpha]$ is integrally closed;
    \item[\textup{(2)}] $f(b) \notin M_v^2$ whenever $m \geq 3$ or $a f'(b) \in M_v$;
    \item[\textup{(3)}] If $\min(v(a), v(m)) = 0$, then $m^{mn}\,f\left(h\left(\frac{-a(m-1)}{m}\right)\right) \notin M_v^2$;
    \item[\textup{(4)}] If $v(m)>0$ and $v(a)>1$, then $R_v[\beta]$ is integrally closed, where $\beta$ is a root of $f(x^p+b)$;
   \item[\textup{(5)}]
If $v(m)>0$ and $v(a)=1$, then
$f(x^\ell+b')$ and
\[
\frac{1}{\pi} \left( f(x^m+b)+ax^{m-1}f'(x^m+b) - f(x^\ell+b')^{p^u} \right)
\]
are coprime modulo $M_v$, where $m=p^u\ell$ and
$(b')^{p^u}\equiv b\pmod{M_v}$;
    \item[\textup{(6)}] If $p \mid t$, then $R_v[\gamma]$ is integrally closed, where $\gamma$ is a root of $f(h(x^p))$.
\end{enumerate}
\end{theorem}
\begin{proof}
    Suppose $R_v[\Theta]$ is integrally closed and let $\theta$ be a root of $\cF(x)= f(h(x))$. It follows from Proposition~\ref{pro:5.13} that $R_v[\alpha]$, $R_v[\gamma]$ and $R_v[\theta]$ are integrally closed. In particular, conditions~(1) and (6) holds. Using Proposition 
    \ref{Pro:4.14}, conditions (4) and (5) holds true. Further, using Proposition \ref{Prop:4.13} along with the arguments used in it's proof, (3) holds. Lastly, Proposition \ref{pro:5.7} and the fact that $R_v[\theta]$ is integrally closed implies that (2) holds true.

{}

Conversely, assume that conditions~(1)--(6) holds. We first show that $R_v[\theta]$ is integrally closed. 
If $\min(v(a), v(m)) = 0$, then conditions~(1), (2) and~(3), together with Proposition~\ref{Prop:4.13}, imply that $R_v[\theta]$ is integrally closed. On the other hand, if $\min(v(a), v(m)) > 0$, then conditions~(4) and~(5), together with Proposition~\ref{Pro:4.14}, again yield that $R_v[\theta]$ is integrally closed. Hence, the simple extension generated by $\theta$ over $R_v$ is indeed integrally closed. 

We now proceed to prove that $R_v[\Theta]$ is integrally closed. 
If $v(t)=0$, setting $m=t$, $a=b=0$, $f = \cF$ and $\cF = \fF$ in Proposition \ref{Prop:4.13}, we have $\min \{ v(t), v(a=0) \} = 0$. Thus, we can apply Proposition \ref{Prop:4.13} to conclude that $R_v[\Theta]$ is intagrally closed if and only if $R_v[\theta]$ is integrally closed, $\cF(0) \notin M_v^2$ and $t^{tmn} \cF((-0(t-1)/t)^t) \notin M_v^2$. We have already proved that $R_v[\theta]$ is integrally closed. Note that $\cF(0) = f(h(0)) = f(b) \notin M_v^2$ by condition (2). Also, $t^{tmn} \cF((-0(t-1)/t)^t) = t^{tmn} f(h(0)) = t^{tmn} f(b) \notin M_v^2 $ using condition (2) and $v(t) = 0$. Thus, we have proved that $R_v[\Theta] $ is integrally closed whenever whenever $v(t) = 0$.

Now suppose that $v(t)>0$, then we have $p \mid t$. Write $t = p^u \ell$ where $p \nmid \ell$. Setting $m=t$, $a=b=0$ $f = \cF$ and $\cF = \fF$ in Proposition~\ref{Pro:4.14}, we have $a = 0 \in M_v^2$. Thus, applying Proposition \ref{Pro:4.14} we conclude that $R_v[\Theta]$ is integrally closed if and only if the simple extension generated by the root of $\cF(x^p) = f(h(x^p))$ is integrally closed, i.e., $R_v[\gamma]$ is integrally closed, which holds by assumption (6). Therefore, the ring $R_v[\Theta]$ is integrally closed irrespective of whether $v(t) = 0$ or $v(t) \neq 0$. This completes the proof of the converse and hence the theorem.
\end{proof}

\section{Proof over rational number} \label{sec:proof over rationals}
We begin with a simple observation that will be used throughout this section. Whenever we apply a result from Section~\ref{sec:monogenity of composition}, we will use the fact that the two conditions $R_v[\theta]$ is integrally closed and $p\nmid [\mathbb{Z}_K:\mathbb{Z}[\theta]]$ are equivalent. To see this, suppose that $R_v=\mathbb{Z}_{(p)}$ with the standard $p$-adic valuation. The integral closure of $\mathbb{Z}_{(p)}[\theta]$ in $K$ is the localized ring of integers $(\mathbb{Z}_K)_{(p)}$. Thus, $\mathbb{Z}_{(p)}[\theta]$ is integrally closed if and only if $\mathbb{Z}_{(p)}[\theta]=(\mathbb{Z}_K)_{(p)}$. On the other hand, localizing the finite quotient $\mathbb{Z}_K/\mathbb{Z}[\theta]$ at $p$ keeps only its $p$-primary part. This localized quotient is trivial if and only if $p$ does not divide the order of the global quotient, namely the index $[\mathbb{Z}_K:\mathbb{Z}[\theta]]$. Therefore, $R_v[\theta]$ is integrally closed if and only if $p\nmid[\mathbb{Z}_K:\mathbb{Z}[\theta]]$.


\begin{proof}[Proof of Theorem~\ref{thm:2.6}]
We prove it by induction on $N$. We first consider the case $N=1$. By Corollary~\ref{T2.1}, the polynomial $f\circ H$ is monogenic if and only if the following conditions hold:
\begin{enumerate}
    \item $f(x)$ is monogenic;
    \item $f(b) = (f\circ H)(0)$ is square-free whenever $m>2$. Further, if $m=2$ then $p^2 \nmid (f\circ H)(0)$ whenever $p \mid a f'(H(0))$;
    \item $p^2\nmid m^{mn}\,f\!\left(h\!\left(\frac{-a(m-1)}{m}\right)\right)$ whenever $p\nmid\gcd(a,m)$;
    \item $p\nmid\operatorname{ind}(f(x^p+b))$ for every prime $p\mid\gcd(a,m)$;
    \item $p\nmid\operatorname{ind}(f(h(x^p)))$ for every prime $p\mid t$.
\end{enumerate}
We first simplify the fourth condition. Let $p\mid\gcd(a,m)$. By hypothesis, every irreducible factor of $f(x)$ over $\mathbb{F}_p$ has multiplicity at least two. Hence,
\[
f(x)\equiv \prod_{i=1}^{r} g_i(x)^{e_i}\pmod p,
\]
where each $g_i(x)$ is irreducible over $\mathbb{F}_p$ and $e_i\ge2$. Consequently, $f(x)\in\langle p,\,g_i(x)^2\rangle$ for every $1\le i\le r.$ Applying Lemma~\ref{lemma:3.6} with the translation $h(x)=x+b$, we obtain $f(x^p+b)\in\langle p,\,g_i(x+b)\rangle^2$ if and only if $f(x^p)\in\langle p,\,g_i(x)\rangle^2$. Using this equivalence, Theorem \ref{valuation dede} along with the observation that $g_i(x)$ is irreducible over $\F_p$ implies $g_i(x+b)$ is irreducible over $\F_p$, we obtain
\[
p\nmid\operatorname{ind}(f(x^p+b)) \iff p\nmid\operatorname{ind}(f(x^p)) \textup{ for every prime } p \mid \gcd(a,m).
\]
Similarly for a prime divisor $p$ of $t$, applying Lemma~\ref{lemma:3.6} with the polynomial $h(x)$ yields  $f(h(x^p))\in\langle p,\,g_i(h(x))\rangle^2$ if and only if $f(x^p)\in\langle p,\,g_i(x)\rangle^2$. Using this observation, we apply Theorem \ref{valuation dede} and the remark followed by Lemma \ref{lem:5.14} to conclude that  $p \nmid \operatorname{ind} (f(h(x^p)))$ if and only if $p \nmid \operatorname{ind} (f(x^p))$ for every prime $p \mid t$. Hence, the index conditions appearing in $(4)$ and $(5)$ are equivalent to the single condition $p$ does not divide index of $f(x^p)$ for every $p\mid \kappa.$ Now suppose that the result holds for $N-1$. We prove it for $N$. By Corollary~\ref{T2.1}, the polynomial $f\circ H^N$ is monogenic if and only if the following conditions are satisfied:
\begin{enumerate}
    \item $(f\circ H^{N-1})(x)$ is monogenic;
    \item $(f\circ H^N)(0)$ is square-free whenever $m>2$. Further, if $m=2$ then $p^2 \nmid (f\circ H^N)(0)$ whenever $p \mid a (f\circ H^{N-1})'(H(0))$;
    \item \( p^2\nmid m^{mn(tm)^{N-1}} (f\circ H^{N-1})
    \!\left( h\!\left(\frac{-a(m-1)}{m}\right) \right)  \)
    whenever $p\nmid\gcd(a,m)$;
    \item $p\nmid\operatorname{ind} \bigl((f\circ H^{N-1})(x^p+b)\bigr)$ for every prime $p\mid\gcd(a,m)$;
    \item $p\nmid\operatorname{ind} \bigl((f\circ H^{N-1})(h(x^p))\bigr)$ for every prime $p\mid t$.
\end{enumerate}
Following the arguments similar to those used in the case $N=1$, it easily follows that the index conditions in (4) and (5) are respectively equivalent to $p$ does not divide the index of $(f\circ H^{N-1})(x^p)$ for every prime $p\mid t\gcd(a,m)$. Again the use of arguments similar to those used for primes $p \mid t$ in the $N=1$ case will yield that the above is equivalent to $p \nmid \operatorname{ind} f(x^p)$ for all primes $p \mid \kappa$. Therefore, the hypotheses of the induction step apply to the polynomial $f\circ H^{N-1}$. Combining the induction hypothesis with conditions (2) and (3), we conclude that $f\circ H^N$ is monogenic if and only if the stated conditions of the theorem hold. This completes the proof.
\end{proof}

\begin{proof}[Proof of Theorem~\ref{thm:2.8}]
Let $p$ be a prime divisor of $m$. By assumption, $p\mid \ell$, so we may write $\ell=p^us$, where $\gcd(p,s)=1.$ A simple Binomial expansion produces the equivalence
\[
F(x)=f(x^\ell)\equiv f(x^s)^{p^u} \pmod{p}.
\]
Hence every irreducible factor of $F(x)$ over $\mathbb{F}_p$ occurs with multiplicity at least two for each prime $p\mid m$. Applying Corollary~\ref{cor:2.7} to the polynomial $F(x)$, we deduce that $\mathcal{F}(x)$ is monogenic if and only if the following conditions hold:
\begin{enumerate}
    \item[(1)] $F(x)=f(x^\ell)$ is monogenic;
    \item[(2)] $F(b)=f(b^\ell)$ is square-free;
    \item[(3)] $p$ does not divide $\operatorname{ind}(F(x^p)) = \op{ind}(f(x^{p\ell}))$ for every prime $p\mid m$.
\end{enumerate}
We now analyze condition~(i). Applying Corollary~\ref{cor:2.7} once again, this time to $F(x)=f(x^\ell)$, shows that $F(x)$ is monogenic if and only if
\begin{enumerate}
    \item[(a)] $f(x)$ is monogenic;
    \item[(b)] $f(0)$ is square-free;
    \item[(c)] $p\nmid\operatorname{ind}(f(x^p))$ for every prime $p\mid\ell$.
\end{enumerate}
Let $p$ be a prime divisor of $m$. It is easy to note that every irreducible factor of $f(x^p)$ modulo $p$ has multiplicity at least two. Using arguments similar to those used in obtaining \eqref{eqn:3.16} with $\tilde{f}(x) = f(x^p)$ over there, we conclude $p\nmid\operatorname{ind}\!\bigl(f(x^{p\ell})\bigr)$ if and only if $p\nmid\operatorname{ind}\!\bigl(f(x^p)\bigr)$ for every $p\mid\ell.$ Substituting this equivalence into condition~(3) and combining it with conditions~(a) -- (c) and~(2), we obtain precisely the stated characterization. This completes the proof.
\end{proof}

The following theorem provide necessary and sufficient condition for monogenity of binomial, which will be used in proof of Theorem~\ref{thm:2.9}.
\begin{theorem}\cite[Theorem~1.3]{Jh-Kh}\label{thm:3.1}
Let $K=\mathbb{Q}(\theta)$ be a Kummer extension of $\mathbb{Q}$, where $\theta$ is a root of the irreducible polynomial $x^n-a \in \mathbb{Z}[x].$ Then the following statements are equivalent:
\begin{enumerate}
    \item[\textup{(i)}] $\mathbb{Z}_K=\mathbb{Z}[\theta]$.
    
    \item[\textup{(ii)}] $a$ is a square-free integer and, for every prime divisor $p$ of $n$, $ p^2 \nmid \bigl(a^p-a\bigr)$.
\end{enumerate}
\end{theorem}

\begin{proof}[Proof of Theorem~\ref{thm:2.9}]
  We will prove by induction on $N.$ We first prove for $N=1$. Let $p$ be  a prime divisor of $m$. Then by hypothesis, $p\mid n$. Write $n=p^{\ell}s$, where $p\nmid s$. Then we have \[f(x)\equiv(x^{s}+a)^{p^{\ell}}\pmod p.\] Hence for every prime $p$, dividing $m$, every irreducible factor of $f(x)$ has multiplicity atleast $2$ over $\F_p$. By Corollary~\ref{cor:2.7}, we have $f\circ g$ is monogenic if and only if $f(x)$ is monogenic, $f(b)$ is square free and $p\nmid\operatorname{ind}f(x^p)$ whenever $p\mid m$. We now further reduce these condition to obtain the desired result. From Theorem~\ref{thm:3.1}, $f(x)$ is monogenic if and only if, $-a$ is square free, and for every $q\mid n$, $q^2\nmid ((-a)^q+a)$. Now we reduce the condition $p\nmid\operatorname{ind}f(x^p)$. Write $f(x^p)=x^{np}+a$, then again applying Theorem~\ref{thm:3.1}, we have $f(x^p)$ is monogenic if and only if, $-a$ is squre free, and $q^2\nmid ((-a)^q+a)$ whenever $q \mid {pn}$. Keeping in mind $\operatorname{rad}(m)\mid n$ and above arguments we get the desired results in $N=1$ case.
  
  Now suppose the result is true for $N=k-1$. We show that it holds for $N=k$. Using same argument as used in step $N=1$, we can show that for every prime divisor $p$ of $m$, any repeated factor of $(f\circ g^{k-1})(x)$ has multiplicity $2$ at least. Hence, by Corollary~\ref{cor:2.7}, $(f\circ g^{k})(x)$ is monogenic if and only if, $(f\circ g^{k-1})(x)$ is monogenic, $(f\circ g^{k-1})(b)=(f\circ g^k)(0)$ square free, and $p\nmid \operatorname{ind}( f\circ g^{k-1})(x^p)$ whenever $p\mid m$.  We first reduce the condition $p\nmid \operatorname{ind}( f\circ g^{k-1})(x^p)$ further. For each prime $p$ divides $m$, write $n=p^us$, where $p\nmid s$ and $x^s+a=\prod_{i=1}^rg_i(x)^{e_i}$, where $e_i\geq 1$ and each $g_i(x)$ are irreducible over $\F_p$. Then, we have 
  \[ f(x) \equiv (x^s+a)^{p^u} \equiv \left(\prod_{i=1}^rg_i(x)^{e_i}\right)^{p^u} \pmod{p}.\] 
  By above congruence, we conclude that $f(x)\in\langle p,\,g_i(x)^2\rangle$ for every $1\leq i\leq r$. Applying Lemma~\ref{lemma:3.6} with $h(x)=g^{k-1}(x)$, we deduce that for each $i\in\{1,2,\ldots,r\}$,
  \begin{equation}\label{eqn:3.1}
      f(x^p)\in\langle p,\, g_i(x)\rangle^2\iff (f\circ g^{k-1})(x^p)\in\langle p,\, g_i\circ g^{k-1}(x)\rangle^2.
  \end{equation}
  It is easy to observe that 
  \[(f\circ g^{k-1})(x^p)\equiv\left(\prod_{i=1}^r((g_i\circ g^{k-1})(x))^{e_i}\right)^{p^u+1}\pmod p.\]
  Thus any irreducible factor of $(f\circ g^{k-1})(x)$ is a factor of some $(g_i\circ g^{k-1})(x)$ for some $i\in\{1,2,\ldots,r\}.$ Thus, applying Theorem \ref{valuation dede} over $\Q$ and remark following Lemma \ref{lem:5.14}, we deduce
  \begin{equation}\label{eqn:3.2}
       p\nmid\operatorname{ind}(f\circ g^{k-1})(x^p) \iff (f\circ g^{k-1})(x^p)\notin\langle p,\, g_i\circ g^{k-1}(x)\rangle^2\quad \forall\, 1\leq i \leq r.
  \end{equation}
 Using the equivalence obtained in \eqref{eqn:3.1} and \eqref{eqn:3.2}, we apply Theorem \ref{valuation dede} to conclude that for every prime divisor $p$ of $m $, $p$ does not divide $ \operatorname{ind} f(x^p)$ if and only if $p$ does not divide $ \operatorname{ind}(f\circ g^{k-1})(x^p)$.
 
 Combining the above arguments, we conclude that $(f\circ g^k)(x^p)$ is monogenic if and only if $(f\circ g^{k-1})(x)$ is monogenic, $(f\circ g^k)(0)$ is square-free, and $p\nmid \operatorname{ind}\bigl(f(x^p)\bigr)$ for every prime $p\mid m$.
By the induction hypothesis, $(f\circ g^{k-1})(x)$ is monogenic if and only if $(f\circ g^i)(0)$ is square-free for all $0\le i\le k-1$ and $p^2 \nmid ((-a)^p+a)$ for every prime $p\mid m$. Combining these conditions with the criterion obtained above yields the desired result. This completes the proof.
\end{proof}
\section{Analytic results} \label{sec:analytic}
The following theorem proved by Granville \cite[Theorem 1]{granville} is the key ingredient in our proof of Theorem \ref{analytic}.
\begin{theorem} \label{thm:4.1}
	Suppose that $f(x) \in \Z[x]$, without any repeated roots. Let $ D=\gcd\{f(n): n \in \Z \}$, and select $D'$ to be the smallest divisor of $D$ for which $D/D'$ is square-free. If the $abc$-conjecture is true, then 
	\[ |\{ 1 \leq n \leq X : f(n)/D' \text{ is square-free.} \}| \sim \lambda_f X \]
	where $\lambda_f>0$ is a positive constant, which we determine as follows:
	\[ \lambda_f = \prod_{p \text{ prime}} \left( 1- \frac{\omega_f(p)}{p^{\,2+ \mathfrak{q}_p}} \right) \]
	where, for each prime $p$, we let $\mathfrak{q}_p$ be the largest power of $p$ which divides $D'$ and let $\omega_f(p)$ denote the number of integers $a$ in the range $1 \leq a \leq p^{\,2+ \mathfrak{q}_p}$ for which $f(a) \equiv 0 \pmod{p^2}$.
\end{theorem}

\begin{proof}[Proof of Theorem~\ref{analytic}]
We begin the proof with a simple observation that
$h(-a(m-1)/m) = b + \varepsilon/m^m$. Since $q \mid a$ and $m \ge 3$, we have
\( \nu_q(\varepsilon) = \nu_q\big(a^m(m-1)^{m-1}\big)
\ge m \ge 3. \) Moreover, since $\gcd(a,m)=1$ and $q\mid a$, we have $q\nmid m$.

Now we will show that $R_f$ is separable, has no fixed prime divisor and its $\gcd$ is squarefree. We can then apply Granville's theorem to obtain infinitely many integers $y$ for which $R_f(y)$ is squarefree. For such a $y$, setting $b=qy$ will give the squarefreeness conditions required in Corollary \ref{T2.1}. We first establish the following four claims.

\noindent \textbf{\underline{Claim 1}:} $R_f(y) \in \mathbb{Z}[y]$ and $R_f(y) \not\equiv 0 \pmod q$ for every $y \in \mathbb{Z}$. \\
Let $y$ be an integer and write
$f(qy) = \sum_{j=0}^n c_j q^j y^j$. Since $f$ is monic and $q$-Eisenstein, we have
$q\mid c_j$ for $0\le j\le n-1$ and $\nu_q(c_0)=1$. Thus, the constant term of $f(qy)/q$ is $c_0/q$, which is an integer not divisible by $q$. For $j\ge 1$, the coefficient $c_jq^j$ is divisible by $q^2$. Therefore,
\begin{equation*}
    \frac{f(qy)}{q} \in \mathbb{Z}[y]
    \quad \text{and} \quad
    \frac{f(qy)}{q} \equiv \frac{c_0}{q} \not\equiv 0 \pmod q.
\end{equation*}
For the second factor, expanding
$ \mathcal{N}_f(qy)
= \sum_{j=0}^n c_j(qm^m y+\varepsilon)^j m^{m(n-j)} $
shows that for $z>0$, each term in the coefficient of $y^z$ contains the factor $c_jq^z(m^m)^z$. Hence, all non-constant terms are divisible by $q^2$. The constant term is
$\mathcal{N}_f(0)
= \sum_{j=0}^n c_j\varepsilon^j m^{m(n-j)}$.
Since $q\nmid m$, we have
$\nu_q(c_0m^{mn})=1$. For $1\le j\le n-1$,
$\nu_q\big(c_j\varepsilon^j m^{m(n-j)}\big)
\ge 1+3=4$, while for $j=n$, $\nu_q(\varepsilon^n)\ge 3n\ge 6$.
It follows that $\nu_q(\mathcal{N}_f(0))=1$. Hence,
\begin{equation*}
    \frac{\mathcal{N}_f(qy)}{q} \in \mathbb{Z}[y]
    \quad \text{and} \quad
    \frac{\mathcal{N}_f(qy)}{q}
    \equiv \frac{\mathcal{N}_f(0)}{q}
    \not\equiv 0 \pmod q.
\end{equation*}
Therefore, $R_f(y)\in\mathbb{Z}[y]$ has degree $2n$ with leading coefficient
$q^{2n-2}m^{mn}$. Moreover, for every integer $y$, its value has a non-zero remainder modulo $q$. 

\noindent \textbf{\underline{Claim 2}:} $\nu_q(f(b)) = \nu_q(\mathcal{N}_f(b)) = 1$ for every $b$ divisible by $q$. \\
Substituting $b=qy$ into the results of Claim 1 gives $f(qy)/q \not\equiv 0 \pmod q$ and $\mathcal{N}_f(qy)/q \not\equiv 0 \pmod q$. Thus, multiplying by $q$ shows that both $f(b)$ and $\mathcal{N}_f(b)$ are divisible by $q$ but not by $q^2$. Therefore, their $q$-adic valuations are exactly $1$. \\[2mm]
\noindent \textbf{\underline{Claim 3}:} $\cF_b(x) := f(h_b(x))$ is irreducible for every $b$ divisible by $q$.\\
Since $q\mid a$ and $q\mid b$, the inner trinomial satisfies $h_b(x)\equiv x^m\pmod q$. Also, since $f(x)$ is $q$-Eisenstein, we have $f(x)\equiv x^n\pmod q$. Therefore, we have
$\cF_b(x)\equiv (x^m)^n\equiv x^{mn} \pmod{q}$. Moreover, the constant term $\cF_b(0)=f(b)$ satisfies $\nu_q(f(b))=1$ by Claim 2. Hence, $\cF_b(x)$ is also $q$-Eisenstein and is therefore irreducible over $\mathbb{Q}$. \\[2mm]
\noindent \textbf{\underline{Claim 4}:} $R_f$ has no repeated roots. \\
Let $\theta_1,\dots,\theta_n$ be the distinct roots of the irreducible polynomial $f$. Thus, the roots of $f(qy)/q$ are $\theta_i/q$. Since $\mathcal{N}_f(b)=m^{mn}f(b+\epsilon/m^m)$, the roots of $\mathcal{N}_f(qy)/q$ are $(\theta_j-\epsilon/m^m)/q$. Hence, both $f(qy)/q$ and $\mathcal{N}_f(qy)/q$ have no repeated roots. Therefore, $R_f$ has a repeated factor if and only if $f(qy)/q$ and $\mathcal{N}_f(qy)/q$ have a common root. Suppose that they have a common root, say $\theta_i/q=(\theta_j-\epsilon/m^m)/q$, i.e., $\theta_j-\theta_i=\epsilon/m^m$. The left-hand side is an algebraic integer, and hence the rational number $\epsilon/m^m$ must be an integer. However, $\gcd(a,m)=1$ and $\gcd(m-1,m)=1$ giving $\gcd(\epsilon,m)=1$. Since $m\neq \pm 1$, we have $\epsilon/m^m\notin\mathbb{Z}$, which is a contradiction. Thus, $R_f$ is separable. \\[2mm]
We now prove that $\gcd\{R_f(y):y\in\mathbb{Z}\}$ is squarefree. Let $p$ be a prime dividing this greatest common divisor. By Claim 1, $R_f(y)\not\equiv 0\pmod q$, and hence $p\neq q$. If $p>2n$, then $\mathbb{F}_p$ contains more elements than the degree of $R_f(y)$. Since $R_f(y)$ is zero modulo $p$ for every $y\in\mathbb{Z}$, it must be the zero polynomial modulo $p$. Thus, $p$ divides the leading coefficient $q^{2n-2}m^{mn}$. Since $p\neq q$, this implies $p\mid m$. Suppose now that $p\mid m$. Reducing $\mathcal{N}_f(qy)/q$ modulo $p$, all terms containing $m$ vanish, and hence $\mathcal{N}_f(qy)/q\equiv\epsilon^n/q\pmod p$. Since $\gcd(\epsilon,m)=1$, we have $p\nmid\epsilon$. Therefore, $\mathcal{N}_f(qy)/q$ does not vanish modulo $p$. This implies $R_f(y)$ is divisible by $p$ only if $p$ divides $f(qy)/q$ for all $y$, i.e., $f(qy)/q$ is identically zero modulo $p$ as $p > 2n > \deg(f)$. But the leading coefficient of $f(qy)/q$ is $q^{n-1}$ which is clearly not divisible by $p$,  giving a contradiction. Thus, $p\nmid m$. Consequently, any fixed prime divisor of $\gcd$ of $R_f(y)$ must satisfy $p\le 2n$, $p\neq q$, and $p\nmid m$. By our hypothesis, $p^2 \nmid R_f(0)$ for every prime $p\le 2n$ with $p\neq q$. Hence, no prime can divide all the values of $R_f(y)$, and therefore $\gcd\{R_f(y):y\in\mathbb{Z}\}$ is squarefree.

Since $R_f$ has no repeated roots and $\gcd\{R_f(y):y\in\mathbb{Z}\}$ is squarefree, we apply Theorem \ref{thm:4.1} with $D'=1$ and $\mathfrak{q}_p=0$. Assuming the $abc$-conjecture, the number of integers $y\le Z$ for which $R_f(y)$ is squarefree is asymptotic to $\lambda Z$, where $\lambda=\prod_p(1-\omega(p)/p^2)$. The positivity of $\lambda$ follows since, for primes not dividing the discriminant, the roots modulo $p$ lift uniquely to roots modulo $p^2$ by Hensel's lemma. Thus, for all sufficiently large $p$, we have $\omega(p)\le 2n<p^2$, and hence $\lambda>0$.

Choose an integer $y$ such that $R_f(y)$ is squarefree, and put $b=qy$. We now show that $f(h_b(x))$ is monogenic. Suppose that an odd prime $p\neq q$ satisfies $p^2\mid f(b)$. Then $p^2\mid f(qy)/q$, and hence $p^2\mid R_f(y)$, which is a contradiction. Together with $\nu_q(f(b))=1$ from Claim 2, this shows that $f(b)$ is squarefree. A similar argument shows that $\mathcal{N}_f(b)$ is squarefree. We now check the conditions of Corollary \ref{T2.1} for $\cF_b(x)=f(h_b(x))$ with $t=1$. The polynomial $\cF_b$ is irreducible by Claim 3. The conditions concerning prime divisors of $\gcd(a,m)$ are vacuous since $\gcd(a,m)=1$. Moreover, $f$ is monogenic, and $f(b)$, $\mathcal{N}_f(b)$ are squarefree, i.e., Conditions 1, 2 and 3 are satisfied. Therefore, $\cF_b$ is monogenic by Corollary \ref{T2.1}. Finally, setting $Y=qZ$ gives the required density lower bound. This completes teh proof.
\end{proof}
\section{Examples} \label{sec:examples}
The following lemma records a useful consequence of the divisibility $q^2\mid\Phi_p(2j)$, which will be used in the proof of the Proposition~\ref{pro:6.2}.

\begin{lemma}
Let $p$ be an odd prime and $j\geq 1$ be an integer. Suppose that $q^2$ divides $\Phi_p(2j)$ for some prime $q$. Then both of the following conditions hold:
\begin{enumerate}
\item $q$ is a generalized Wieferich prime to base $2j$, that is,
$(2j)^{q-1}\equiv 1\pmod{q^2}$;
\item $p$ divides $q-1$.
\end{enumerate}
\end{lemma}

\begin{proof}
Assume that $q^2\mid\Phi_p(2j)$ for some prime $q$. Since $\Phi_p(x)=\frac{x^p-1}{x-1}$, we have $(2j)^p-1=(2j-1)\Phi_p(2j)$. Thus, $q^2\mid\Phi_p(2j)$ implies
\begin{equation}
(2j)^p\equiv 1\pmod{q^2}.
\end{equation}
Therefore, $2j$ is invertible modulo $q^2$, and its multiplicative order modulo $q^2$ divides $p$. Since $p$ is prime, the order of $2j$ modulo $q^2$ is either $1$ or $p$. Suppose to the contrary, that the order is $1$. Then $2j\equiv1\pmod{q^2}$. Hence,
\begin{equation}
\Phi_p(2j)=(2j)^{p-1}+(2j)^{p-2}+\cdots+1\equiv p\pmod{q^2}.
\end{equation}
On the other hand, $q^2\mid\Phi_p(2j)$ gives $\Phi_p(2j)\equiv0\pmod{q^2}$. Therefore, $p\equiv0\pmod{q^2}$, which is not possible since $p$ is prime. Thus, the order of $2j$ modulo $q^2$ must be $p$. We now show that $p$ divides $q-1$. By Lagrange's theorem, the order of an element of $(\mathbb{Z}/q^2\mathbb{Z})^\times$ divides the order of the group. Since $|(\mathbb{Z}/q^2\mathbb{Z})^\times|=\phi(q^2)=q(q-1)$ and the order of $2j$ is $p$, we have $p\mid q(q-1).$ As $p$ is prime, it follows that either $p=q$ or $p\mid(q-1)$.

We now show that $p\neq q$. Suppose that $p=q$. Since the order of $2j$ modulo $q^2$ is $q$, we have $(2j)^q\equiv1\pmod{q^2}$, and hence $2j\equiv1\pmod q$. Write $2j=1+kq$ for some integer $k$. Then

$$
\Phi_q(1+kq)=\sum_{r=0}^{q-1}(1+kq)^r\equiv q+kq^2\frac{q-1}{2}\equiv q\pmod{q^2}.
$$

But $q^2\mid\Phi_q(2j)$, so $\Phi_q(2j)\equiv0\pmod{q^2}$, giving $q\equiv0\pmod{q^2}$, a contradiction. Thus $p\neq q$, and consequently $p\mid(q-1)$. Since $p\mid(q-1)$, there exists an integer $k$ such that $q-1=kp$. From the initial divisibility assumption, we have $(2j)^p\equiv1\pmod{q^2}$. Raising both sides to the $k$-th power gives
\begin{equation}
(2j)^{kp}\equiv1\pmod{q^2}.
\end{equation}
Since $kp=q-1$, it follows that
\begin{equation}
(2j)^{q-1}\equiv1\pmod{q^2}.
\end{equation}
Thus, $q$ is a generalized Wieferich prime to base $2j$, proving condition (1).
\end{proof}
\begin{lemma}\label{lem:6.2}
Let $p$ be an odd prime and  $\Phi_p(x)$ denote the $p$-th cyclotomic polynomial. Then
\[
\Phi_p(x)\mid\frac{\Phi_p(x^2)-\Phi_p(x)^2}{2}
\]
in $\mathbb{Z}[x]$.
\end{lemma}

\begin{proof}
Since $p$ is odd, $\Phi_p(x)=1+x+\cdots+x^{p-1}$. Moreover,
$\Phi_p(x^2)=\frac{x^{2p}-1}{x^2-1}
=\Phi_p(x)\frac{x^p+1}{x+1}$.
Thus
\[
\Phi_p(x^2)-\Phi_p(x)^2
=\Phi_p(x)\left(\frac{x^p+1}{x+1}-\Phi_p(x)\right).
\]
Since $p$ is odd,
$\frac{x^p+1}{x+1}=x^{p-1}-x^{p-2}+\cdots-x+1$, and hence
$\frac{x^p+1}{x+1}-\Phi_p(x)
=-2(x^{p-2}+x^{p-4}+\cdots+x)$.
Therefore
\[
\frac{\Phi_p(x^2)-\Phi_p(x)^2}{2}
=-\Phi_p(x)(x^{p-2}+x^{p-4}+\cdots+x)\in\mathbb{Z}[x].
\]
Hence $\Phi_p(x)\mid\frac{\Phi_p(x^2)-\Phi_p(x)^2}{2}$. This completes the proof of lemma.
\end{proof}
\begin{proposition}\label{pro:6.2}
Let $p$ be an odd prime and $j\geq1$ an integer. Let $\Phi_p(x)$ denote the $p$-th cyclotomic polynomial, which is  monogenic. Define the Misiurewicz trinomial $$h(x)=x^2-4jx+(4j^2+2j).$$ If $\Phi_p(2j)$ is squarefree and $j$ is odd. Then the infinite dynamical tower defined by $\mathfrak{F}_k(x)=\Phi_p(h^k(x))$ is monogenic for every $k\geq1$.
\end{proposition}
\begin{proof}
By classical algebraic  number theory, the ring of integers of the cyclotomic field $\mathbb{Q}(\zeta_p)$ is  $\mathbb{Z}[\zeta_p]$. Hence, the base polynomial $f(x) = \Phi_p(x)$ is monogenic.

Now we verify condition~(ii) of Corollary~\ref{T2.1}. Since $m=2$ in our case, it suffices to show that $q^2\nmid\phi_p(b)$ whenever $q\mid a\phi_p'(b)$. There are two cases,  either $q\mid a(=2j)$ or $q\mid\phi_p'(b)$. First, suppose that $q\mid2j$. Since $b=2j(2j+1)$, we have $q\mid b$. Therefore, $\phi_p(b)=b^{p-1}+b^{p-2}+\cdots+b+1\equiv1\pmod q$. Hence $q\nmid\phi_p(b)$, and consequently $q^2\nmid\phi_p(b)$.

It remains to consider the case $q\mid\phi_p'(b)$. If $q\nmid\phi_p(b)$, then there is nothing to prove. Thus, suppose that $q\mid\phi_p(b)$ and $q\mid\phi_p'(b)$. From the identity $(x-1)\phi_p(x)=x^p-1,$ we obtain, upon differentiating, $\phi_p(b)+(b-1)\phi_p'(b)=pb^{p-1}$. It follows that $q\mid pb^{p-1}$. Since we have already proved that $q\nmid b$, we must have $q=p$.

We now show that $p^2\nmid\phi_p(b)$. First, we claim that $p\mid\phi_p(b)$ implies $b\equiv1\pmod p$. Indeed, from $(b-1)\phi_p(b)=b^p-1$ and Fermat's little theorem, $b^p\equiv b\pmod p$, we obtain $(b-1)\phi_p(b)\equiv b-1\pmod p$, or equivalently, $(b-1)(\phi_p(b)-1)\equiv0\pmod p$. Since $p\mid\phi_p(b)$, we have $p\nmid(\phi_p(b)-1)$. Hence, $b\equiv1\pmod p$. Writing $b=1+kp$, we obtain
\[
\phi_p(b)=\frac{(1+kp)^p-1}{kp}
=p+\binom{p}{2}kp+\binom{p}{3}k^2p^2+\cdots+k^{p-1}p^{p-1}.
\]
Since $p\mid\binom{p}{r}$ for $1\le r\le p-1$, all terms after the first are divisible by $p^2$. Therefore, $\phi_p(b)\equiv p\pmod{p^2}$, and hence $p^2\nmid\phi_p(b)$. Since $q=p$, this proves that $q^2\nmid\phi_p(b)$. Thus condition~(ii) of Corollary~\ref{T2.1} is satisfied.

The critical point of $h(x)$ is $x=2j$. Evaluating $h$ at this point gives $h(2j)=2j.$ Thus, $2j$ is a fixed point of $h$, and hence $h^k(2j)=2j$ for every $k\geq1$. Therefore, along the critical orbit, we have $\Phi_p(h^k(2j))=\Phi_p(2j).$ By hypothesis, $\Phi_p(2j)$ is squarefree. Condition~(iii) of Corollary~\ref{thm:2.1} is satisfied.

Since $\gcd (a, m)=2$, so for condition $(iv)$, we need to show that $2$ does not divide index of $\phi_p(x^2+b)$. Using $2\mid b$, we have
\begin{align*}
    \phi_p(x^2+b)=& \sum_{i=1}^{p-1}(x^2+b)^i+1\\
    =& \sum_{i=1}^{p-1}(x^{2i}+ix^{2(i-1)}b)+1+4A(x)\\
    =& \phi_p(x^2)+b\phi'_p(x^2)+4B(x)
\end{align*}
for some $A(x), B(x)\in\Z[x]$. Let $\phi_p(x)\equiv\prod g_i(x)\pmod 2$. Write $\phi_p(x)=\prod g_i(x)+2H(x)$ for some $H(x)\in\Z[x]$. Since $2\mid b$, we have $\phi_p(x^2+b) \equiv (\prod g_i(x))^2\pmod 2$.  Using this and previous equation, we have
\[\frac{\phi_p(x^2+b)-(\prod g_i(x))^2}{2}=\frac{\phi_p(x^2)-\phi_p(x)^2}{2}+\frac{b}{2}\phi_p'(x^2)+2C(x)\]
for some $C(x)\in\Z[x]$. Using Lemma~\ref{lem:6.2} and above equation, we obtain
\begin{equation}\label{eqn:6.5}
\frac{\phi_p(x^2+b)-\left(\prod_i g_i(x)\right)^2}{2}
\equiv j\phi_p'(x^2)
\pmod{\langle 2,\phi_p(x)\rangle}.
\end{equation}
Let $\alpha$ be a root of $\Phi_p(x)$ in $\overline{\mathbb F}_2$. Since
$p$ is odd, $\alpha$ is a primitive $p$-th root of unity. As $\gcd(2,p)=1$, $\alpha^2$ is also a primitive $p$-th root of unity,
and hence $\Phi_p(\alpha^2)=0$. Moreover, $x^p-1$ is separable over $\mathbb F_2$, since
$(x^p-1)'=px^{p-1}=x^{p-1}\neq0$. Thus $\Phi_p(x)$ is separable. Therefore $\Phi_p'(\alpha^2)\neq0$. Hence $\alpha$ is not a root of $\Phi_p'(x^2)$, and consequently
\[
\gcd\bigl(\Phi_p(x),\Phi_p'(x^2)\bigr)=1
\quad\text{in }\mathbb F_2[x].
\]
By Dedekind theorem and \eqref{eqn:6.5} we get desired result whenever $j$ is odd.
\end{proof}

\bibliographystyle{amsplain}
\bibliography{references}  

\end{document}